\documentclass[a4paper,12pt,reqno]{amsart}
\usepackage{amssymb, amsmath}
\usepackage{amssymb, amsmath,amsthm}
\usepackage{amssymb, amsmath}
\usepackage{graphics}
\usepackage[dvips]{graphicx}
\usepackage{eepic}
\usepackage{epic}
\usepackage{enumerate}
\newtheorem{thm}{Theorem}[section]
\newtheorem{crl}[thm]{Corollary}

\newtheorem{lmm}[thm]{Lemma}

\newtheorem{prp}[thm]{Proposition}

\newtheorem{theorem-definition}[thm]{Theorem-Definition}
\theoremstyle{definition}
\newtheorem{dfn}[thm]{Definition}
\newtheorem{exa}[thm]{Example}
\theoremstyle{remark}
\newtheorem{rem}{Remark}
\allowdisplaybreaks

\def\aaa{{\widetilde{\mathfrak{B}}}}

\def\aa{{\mathfrak{B}}}

\def\mm{{{\widetilde{\mathfrak{m}}}_c}}
\def\m2{{{\widetilde{\mathfrak{m}}}_2}}

\DeclareFontEncoding{OT2}{}{}
\DeclareFontSubstitution{OT2}{cmr}{m}{n}
\DeclareSymbolFont{cyss}{OT2}{wncyss}{m}{n}
\DeclareMathSymbol{\sh}{\mathbin}{cyss}{`x}

\begin{document}

\baselineskip 20pt
\title[Desingularization of multiple zeta-functions]{Desingularization of multiple zeta-functions of generalized Hurwitz-Lerch type and evaluation of $p$-adic multiple $L$-functions at arbitrary integers}
\author[H. Furusho]{Hidekazu Furusho}
\address{H. Furusho:\ Graduate School of Mathematics, 
Nagoya University, Chikusa-ku, Nagoya
464-8602, Japan}
\email{furusho@math.nagoya-u.ac.jp}
\author[Y. Komori]{Yasushi Komori}
\address{Y. Komori:\ Department of Mathematics, 
Rikkyo University, 
Nishi-Ikebukuro, Toshima-ku, 
Tokyo 171-8501, Japan}
\email{komori@rikkyo.ac.jp}
\author[K. Matsumoto]{Kohji Matsumoto}
\address{K. Matsumoto:\ Graduate School of Mathematics, 
Nagoya University, Chikusa-ku, Nagoya
464-8602, Japan}
\email{kohjimat@math.nagoya-u.ac.jp}
\author[H. Tsumura]{Hirofumi Tsumura}
\address{H. Tsumura:\ Department of Mathematics and Information Sciences, Tokyo Metropolitan University, 1-1, Minami-Ohsawa, Hachioji, Tokyo 192-0397 Japan}
\email{tsumura@tmu.ac.jp}
\subjclass[2010]{11M32, 11S40, 11G55}

\keywords{\textit{multiple zeta-function, multiple polylogarithm, 
desingularization, $p$-adic multiple $L$-function, $p$-adic multiple polylogarithm}}         

%

\begin{abstract}      
We study analytic properties of multiple zeta-functions of generalized 
Hurwitz-Lerch type. First, as a special type of them, we consider multiple zeta-functions of 
generalized Euler-Zagier-Lerch type and investigate their analytic properties which were already announced in our previous paper. 
Next we give `desingularization' of multiple zeta-functions of generalized 
Hurwitz-Lerch type, which include those of generalized Euler-Zagier-Lerch type, the Mordell-Tornheim type, and so on. As a result, 
the desingularized multiple zeta-function turns out to be an entire function and can be expressed as a finite sum of ordinary multiple zeta-functions of the same type. 
As applications, we explicitly compute special values of desingularized double zeta-functions of Euler-Zagier type.
We also extend our previous results concerning a relationship 
between $p$-adic multiple $L$-functions and $p$-adic multiple star polylogarithms 
to more general 
indices with arbitrary (not necessarily all positive) integers.
\end{abstract}
\maketitle

\tableofcontents
\setcounter{section}{-1}
\section{Introduction}


In the present paper we continue our study developed in our previous papers \cite{FKMT1,FKMT2},
with supplying some proofs of results in \cite{FKMT1} which were stated with no proof.  
In \cite{FKMT1},
we studied multiple zeta-functions of generalized Euler-Zagier-Lerch 
type (see below) and considered their analytic properties. 
Based on those considerations, 
we introduced 
the method of \textit{desingularization} of multiple zeta-functions, which is to resolve all singularities of them.
By this method we constructed the desingularized multiple zeta-function which is entire and can be expressed as a finite sum of ordinary multiple zeta-functions.



The first main purpose of the present paper is to extend our theory of
desingularization to the following more general situation.

Let $\xi_k,\gamma_{jk},\beta_j$ ($1\leqslant  j\leqslant  d,1\leqslant  k\leqslant  r$) 
be complex parameters with $|\xi_k|\leqslant  1$, real parts $\Re \gamma_{jk}\geqslant  0$, 
$\Re \beta_j >0$, and let $s_j$ 
($1\leqslant  j\leqslant  d$) be complex variables.   We assume that for each $j$ 
($1\leqslant  j\leqslant  d$), at least one of $\Re \gamma_{jk}>0$.
We define the \textbf{multiple zeta-functions of generalized Hurwitz-Lerch type}
by
\begin{align}
  \label{eq:L_series}
  &\zeta_r((s_j);(\xi_k);(\gamma_{jk});(\beta_j)) \\
  &=
  \sum_{m_1=0}^\infty\cdots\sum_{m_r=0}^\infty
  \frac{\xi_1^{m_1}\cdots \xi_r^{m_r}}
  {(\beta_1+\gamma_{11}m_1+\cdots+\gamma_{1r}m_r)^{s_1}\cdots (\beta_d+\gamma_{d1}m_1+\cdots+\gamma_{dr}m_r)^{s_d}}
  \notag\\
  &=
  \sum_{m_1=0}^\infty\cdots\sum_{m_r=0}^\infty
  \frac{\prod_{k=1}^r \xi_k^{m_k}}
  {\prod_{j=1}^d(\beta_j+\sum_{k=1}^r \gamma_{jk}m_k)^{s_j}}.\notag
\end{align}
Obviously this is convergent absolutely when $\Re s_j>1$ for $1\leqslant  j\leqslant  d$, and
it is known that this can be continued meromorphically to the whole space
$\mathbb{C}^d$ (see \cite{Kom10}).

In the present paper we will construct desingularized multiple zeta-functions,
which will be expressed as a finite sum of 
$\zeta_r((s_j);(\xi_k);(\gamma_{jk});(\beta_j))$.

The \textbf{multiple zeta-function of generalized Euler-Zagier-Lerch            
type} defined by
\begin{align}\label{Def-EZL-zeta}                                                       
\zeta_r((s_j);(\xi_j);(\gamma_j))=\sum_{m_1=1}^{\infty}\cdots                           
\sum_{m_r=1}^{\infty}\prod_{j=1}^r \xi_j^{m_j}(m_1\gamma_1+\cdots+m_j\gamma_j)          
^{-s_j}                                                                                
\end{align}
for parameters $\xi_j,\gamma_{j}\in \mathbb{C}$ ($1\leqslant  j\leqslant  r$) with $|\xi_j|= 1$ and 
$\Re\gamma_j>0$, is a special case of \eqref{eq:L_series}.
In fact, putting $d=r$, $\gamma_{jk}=\gamma_k$ ($j\geqslant  k$), 
$\gamma_{jk}=0$ ($j<k$), and
$\beta_j=\gamma_1+\cdots+\gamma_j$, \eqref{eq:L_series} reduces to
\eqref{Def-EZL-zeta}.
This \eqref{Def-EZL-zeta} was the main actor of 
the previous paper \cite{FKMT1}.

When $\xi_j=\gamma_j=1$ for all $j$, \eqref{Def-EZL-zeta} is the famous Euler-Zagier
multiple sum (Hoffman \cite{Hof92}, Zagier \cite{Zag94}):
\begin{align}                                                                           
& \zeta_r(s_1,\ldots,s_r)=\sum_{m_1=1}^{\infty}\cdots \sum_{m_r=1}^{\infty}
\prod_{j=1}^r(m_1+\cdots+m_j)^{-s_j}. \label{EZ-zeta}                                   
\end{align}
Singularities of \eqref{EZ-zeta} have been determined explicitly (see Akiyama, Egami and Tanigawa \cite{AET01}).

On the other hand,
when $r=1$ and $\gamma_1=1$, then the above series coincides with the Lerch
zeta-function
\begin{align}\label{Def-Lerch}
\phi(s_1,\xi_1)=\sum_{m_1=1}^{\infty}\xi_1^{m_1}m_1^{-s_1}.
\end{align}
It is known that $\phi(s_1,\xi_1)$ is entire if $\xi_1\neq 1$, while if $\xi_1=1$
then $\phi(s_1,1)$ is nothing but the Riemann zeta-function $\zeta(s_1)$ and has 
a simple pole at $s_1=1$.

%
%
%
%

The plan of the present paper is as follows.

In Section \ref{sec-a} we prove that $\zeta_r((s_j);(\xi_j);(\gamma_j))$
can be continued meromorphically
to the whole space $\mathbb{C}^r$, and its singularities can be explicitly
given (Theorems \ref{thm1} and \ref{thm2}).   This result was announced in \cite[Section 2]{FKMT1}
without proof.
The assertion of the meromorphic continuation is, as mentioned above, already
given in \cite{Kom10}.   However in Section \ref{sec-a} we give an alternative
argument, based on Mellin-Barnes integrals, which is probably more suitable to obtain
explicit information on singularities.


In Section \ref{sec-3}, 
we give desingularization of the multiple zeta-functions of generalized Hurwitz-Lerch type (see \eqref{eq:L_series}), which include those of generalized Euler-Zagier-Lerch type, the Mordell-Tornheim type, and so on. In fact, we will show that these desingularized multiple zeta-functions are entire (see Theorem \ref{thm:ann}), which was already announced in \cite[Remark 4.5]{FKMT1}. Actually this includes our previous result shown in \cite[Theorem 3.4]{FKMT1}. We further show that these desingularized multiple zeta-functions can be expressed as finite sums of ordinary multiple zeta-functions (see Theorem \ref{thm:des}). 

In Section \ref{sec-4}, we give some examples of desingularization of various multiple zeta-functions. The main technique is a certain generalization of ours used in the proof of \cite[Theorem 3.8]{FKMT1}. In particular, we give desingularization of multiple zeta-functions of root systems introduced by the second, the third and the fourth authors (see, for example, \cite{KMT10}). 

In Section \ref{sec-5}, we study special values of desingularized double zeta-functions of Euler-Zagier type. More generally, we give some functional relations for desingularized double zeta-functions and ordinary double zeta-functions of Euler-Zagier type (see Propositions \ref{P-5-2}, \ref{P-5-4} and \ref{P-5-6}). By marvelous cancellations among singularities of ordinary double zeta-functions, we can explicitly compute special values of desingularized double zeta-functions of Euler-Zagier type at any integer points (see Examples \ref{Ex-5-3}, \ref{Ex-5-5}, \ref{Ex-5-7} and Proposition \ref{P-5-9}).

An 
important aspect of \cite{FKMT2} is the construction of the theory of
$p$-adic multiple $L$-functions.   The second main purpose of the present paper is
to give a certain extension of our result on special values of $p$-adic multiple
$L$-functions.

In \cite{KMT11}, the second, the third and the
fourth authors introduced $p$-adic double $L$-functions, as the double analogue
of the classical Kubota-Leopoldt $p$-adic $L$-functions.   In \cite{FKMT2}, we
generalized the argument in \cite{KMT11} to define $p$-adic multiple $L$-functions.
On the other hand, the first author \cite{Furu04} \cite{Furu07} developed the
theory of $p$-adic multiple polylogarithms under a very different motivation.
A remarkable discovery in \cite{FKMT2} is that there is a connection between these
two multiple notions.   In fact, we proved that the values of $p$-adic multiple
$L$-functions at positive integer points can be described in terms of $p$-adic
multiple star polylogarithms (\cite[Theorem 3.41]{FKMT2}).

In Section \ref{pMLF=pMPL} of the present paper, we extend this result to 
obtain the description of the values of $p$-adic multiple $L$-functions at 
arbitrary (not necessarily all positive) integer points in terms of $p$-adic multiple star polylogarithms
(Theorem \ref{L-Li theorem-2}).

\section{The meromorphic continuation and the location of singularities}\label{sec-a}

The purpose of this section is to prove the following result which was announced 
in \cite[Theorem 2.3]{FKMT1}.

\begin{thm}\label{thm1}
The function $\zeta_r((s_j);(\xi_j);(\gamma_j))$ can be continued meromorphically
to the whole space $\mathbb{C}^r$.   Moreover,

{\rm (i)} If $\xi_j\neq 1$ for all $j$ {\rm(}$1\leqslant  j\leqslant  r${\rm)}, then
$\zeta_r((s_j);(\xi_j);(\gamma_j))$ is entire.

{\rm (ii)} If $\xi_j\neq 1$ for all $j$ {\rm(}$1\leqslant  j\leqslant  r-1${\rm)} and
$\xi_r=1$, then $\zeta_r((s_j);(\xi_j);(\gamma_j))$ has a unique simple singular
hyperplane $s_r=1$.

{\rm (iii)} If $\xi_j=1$ for some $j$ {\rm(}$1\leqslant  j\leqslant  r-1${\rm)}, then
$\zeta_r((s_j);(\xi_j);(\gamma_j))$ has infinitely many simple singular
hyperplanes.
\end{thm}

Actually the location of the singular hyperplanes will be more explicitly
described in Theorem \ref{thm2}. 

\begin{rem}\label{Rem-MPL}
The multiple polylogarithm is defined by
\begin{align}\label{MPL-01}
Li_{n_1,\ldots,n_r}(z_1,\ldots,z_r)&=\sum_{0<k_1<\cdots<k_r}\frac{z_1^{k_1}\cdots z_r^{k_r}}{k_1^{n_1}\cdots k_r^{n_r}}\\
&=\sum_{m_1=1}^{\infty}\cdots\sum_{m_r=1}^{\infty}\prod_{j=1}^r
(z_j\cdots z_r)^{m_j}(m_1+\cdots+m_j)^{-n_j},\notag
\end{align}
where $(n_j)\in \mathbb{N}^r$ and $(z_j)\in \mathbb{C}^r$ with $|z_j|= 1$ $(1\leqslant j\leqslant r)$ (see Goncharov \cite{Gon}). Inspired by this definition, we generally define 
\begin{equation}
Li_{s_1,\ldots,s_r}(z_1,\ldots,z_r)=\zeta_r((s_j);(\prod_{\nu=j}^{r}z_\nu);(1)) \label{MPL-02}
\end{equation}
for $(s_j)\in \mathbb{C}^r$ 
and $(z_j)\in \mathbb{C}^r$ with $|z_j|= 1$ $(1\leqslant j\leqslant r)$ 
(see \eqref{Def-EZL-zeta}). In fact, it follows from Theorem \ref{thm1} that the right-hand side of \eqref{MPL-02} can be meromorphically continued to $(s_j)\in \mathbb{C}^r$. Moreover, when $\prod_{\nu=j}^{r}z_\nu\ne 1$ for all $j$, the right-hand side is entire. 
In particular, setting $\xi_j=\prod_{\nu=j}^{r}z_\nu$ $(1\leqslant j\leqslant r)$ and $\xi_{r+1}=1$, we obtain
\begin{equation}
\zeta_r((n_j);(\xi_j);(1))=Li_{n_1,\ldots,n_r}\left(\frac{\xi_1}{\xi_2},\frac{\xi_2}{\xi_3},\ldots ,\frac{\xi_r}{\xi_{r+1}}\right) \label{MPL-03}
\end{equation}
for all $(n_j)\in \mathbb{Z}^r$ when $\xi_j\ne 1$ $(1\leqslant j\leqslant r)$.
In Section \ref{pMLF=pMPL}, we will show a $p$-adic version of \eqref{MPL-03} (see Theorem \ref{L-Li theorem-2} and Remark \ref{Rem-fin}).

\end{rem}

Now we start the proof of Theorem \ref{thm1}.
Let $C(j,r)$ be the number of $h$ ($j\leqslant  h\leqslant  r$) 
for which $\xi_h=1$ holds.
We first prove the following lemma.

\begin{lmm}\label{lmm-a}
The function $\zeta_r((s_j);(\xi_j);(\gamma_j))$ can be continued meromorphically
to the whole space $\mathbb{C}^r$, and its possible singularities can be listed as
follows, where $\ell\in\mathbb{N}_0:=\mathbb{N}\cup\{0\}$.   

$\bullet$ If $\xi_j=1$, then $s_j+s_{j+1}+\cdots+s_r=C(j,r)-\ell$ $\quad$
{\rm(}$1\leqslant  j\leqslant  r-1${\rm)},

$\bullet$ If $\xi_r=1$, then $s_r=1$,

$\bullet$ If $\xi_j\neq 1$ for all $j$ {\rm(}$1\leqslant  j\leqslant  r${\rm)}, then
$\zeta_r((s_j);(\xi_j);(\gamma_j))$ is entire.
\end{lmm}

\begin{proof}
We prove the theorem by induction on $r$.   In the case $r=1$, our zeta-function
is essentially the Lerch zeta-function \eqref{Def-Lerch}, so the assertion of the
lemma is classical.
 
Now let $r\geqslant  2$, and assume that the assertion of the lemma is true for $r-1$.
The proof is based on the Mellin-Barnes integral formula
\begin{align}\label{a1}
(1+\lambda)^{-s}=\frac{1}{2\pi i}\int_{(c)}\frac{\Gamma(s+z)\Gamma(-z)}{\Gamma(s)}
\lambda^z dz,
\end{align}
where $s,\lambda\in\mathbb{C}$, $\Re s>0$, $|\arg\lambda|<\pi$, $\lambda\neq 0$,
$-\Re s<c<0$ and the path of integration is the vertical line $\Re z=c$.
This formula has been frequently used to show the meromorphic continuation of
various multiple zeta-functions (e.g. \cite{Mat02}, \cite{Mat03a}, \cite{Mat03b}).
In particular, the following argument is quite similar to that in \cite{Mat03b}. In what follows, $\varepsilon$ denotes an arbitrarily small positive number, not necessarily the same at each occurrence.

First of all, using \cite[Theorem 3]{Mat02}, we see that series
\eqref{Def-EZL-zeta} is absolutely convergent in the region
\begin{align}\label{a2}
\{(s_1,\ldots,s_r)\;|\;\sigma_{r-j+1}+\cdots+\sigma_r>j\;(1\leqslant  j\leqslant  r)\},
\end{align}
where $\sigma_j=\Re s_j$ ($1\leqslant j\leqslant r$).
At first we assume that $(s_1,\ldots,s_r)$ is in this region.   Divide
\begin{align*}
\lefteqn{(m_1\gamma_1+\cdots+m_r\gamma_r)^{-s_r}}\\
&=(m_1\gamma_1+\cdots+m_{r-1}\gamma_{r-1})^{-s_r}
\left(1+\frac{m_r\gamma_r}{m_1\gamma_1+\cdots+m_{r-1}\gamma_{r-1}}\right)
^{-s_r},
\end{align*}
and apply \eqref{a1} to the second factor on the right-hand side with
$\lambda=m_r\gamma_r/(m_1\gamma_1+\cdots+m_{r-1}\gamma_{r-1})$ to obtain
\begin{align}\label{a3}
\lefteqn{\zeta_r((s_j);(\xi_j);(\gamma_j))}\\
&=\frac{1}{2\pi i}\int_{(c)}\frac{\Gamma(s_r+z)\Gamma(-z)}{\Gamma(s_r)}
\sum_{m_1=1}^{\infty}\cdots\sum_{m_r=1}^{\infty}\frac{\xi_1^{m_1}}{(m_1\gamma_1)^{s_1}}
\times
\cdots\times\frac{\xi_{r-1}^{m_{r-1}}}{(m_1\gamma_1+\cdots+m_{r-1}\gamma_{r-1})
^{s_{r-1}}}\notag\\
&\qquad\times\frac{\xi_{r}^{m_{r}}}{(m_1\gamma_1+\cdots+m_{r-1}\gamma_{r-1})^{s_r}}
\left(\frac{m_r\gamma_r}{m_1\gamma_1+\cdots+m_{r-1}\gamma_{r-1}}\right)^z dz\notag\\
&=\frac{1}{2\pi i}\int_{(c)}\frac{\Gamma(s_r+z)\Gamma(-z)}{\Gamma(s_r)}
\zeta_{r-1}((s_1,\ldots,s_{r-2},s_{r-1}+s_r+z);(\xi_1,\ldots,\xi_{r-1});
(\gamma_1,\ldots,\gamma_{r-1}))\notag\\
&\qquad\times\gamma_r^z\phi(-z,\xi_r)dz,\notag
\end{align}
where $-\sigma_r<c<-1$.   (To apply \eqref{a1} it is enough to assume $c<0$, but
to ensure the convergence of the above multiple series it is necessary to assume
$c<-1$.)

Next we shift the path of integration from $\Re z=c$ to $\Re z=M-\varepsilon$,
where $M$ is a large positive integer, and $\varepsilon$ is a small positive number.    This is possible because, by virtue of
Stirling's formula, we see that the integrand is of rapid decay when $\Im z\to\infty$.
Relevant poles are $z=0,1,2,\ldots$
(coming from $\Gamma(-z)$) and $z=-1$ if $\xi_r=1$ (coming from $\phi(-z,\xi_r)$).
Counting the residues of those poles, we obtain
\begin{align}\label{a4}
\ \lefteqn{\zeta_r((s_j);(\xi_j);(\gamma_j))}\\
&=\delta(r)\frac{\gamma_r^{-1}}{s_r-1}                                             
\zeta_{r-1}((s_1,\ldots,s_{r-2},s_{r-1}+s_r-1);                                         
(\xi_1,\ldots,\xi_{r-1});(\gamma_1,\ldots,\gamma_{r-1}))\notag\\
&+\sum_{k=0}^{M-1}\binom{-s_r}{k}\zeta_{r-1}((s_1,\ldots,s_{r-2},s_{r-1}+s_r+k);
(\xi_1,\ldots,\xi_{r-1});(\gamma_1,\ldots,\gamma_{r-1}))\notag\\
&\qquad\times\gamma_r^k \phi(-k,\xi_r)\notag\\
&+\frac{1}{2\pi i}\int_{(M-\varepsilon)}\frac{\Gamma(s_r+z)\Gamma(-z)}{\Gamma(s_r)}\notag\\
& \qquad \times \zeta_{r-1}((s_1,\ldots,s_{r-2},s_{r-1}+s_r+z);(\xi_1,\ldots,\xi_{r-1});                
(\gamma_1,\ldots,\gamma_{r-1}))\notag\\
&\qquad\times\gamma_r^z\phi(-z,\xi_r)dz \notag\\
&=X+\sum_{k=0}^{M-1}Y(k)+Z,\notag
\end{align}
say, where 
\begin{equation}\label{a-delta} 
  \delta(r)=                                                                            
  \begin{cases}                                                                         
    1\qquad&(\xi_r=1), \\                                                                
    0\qquad&(\xi_r\neq1).                                                               
  \end{cases}                                                                           
\end{equation}
From \eqref{a2} we see that 
$$
\zeta_{r-1}((s_1,\ldots,s_{r-2},s_{r-1}+s_r+z);(\xi_1,\ldots,\xi_{r-1});               
(\gamma_1,\ldots,\gamma_{r-1}))
$$
is absolutely convergent if
$$
\sigma_{r-j}+\cdots+\sigma_r+\Re z>j\qquad(1\leqslant  j\leqslant  r-1),
$$
so the integral $Z$ is convergent
(and hence holomorphic) in the region
\begin{align}\label{a5}
\{(s_1,\ldots,s_r)\;|\;
\sigma_{r-j}+\cdots+\sigma_r>j-M+\varepsilon\quad(0\leqslant  j\leqslant  r-1)\}.
\end{align}
(Here, the condition corresponding to $j=0$ is necessary to assure that
the factor $\Gamma(s_r+z)$ in the integrand does not encounter the poles.)
Therefore by \eqref{a4} and the assumption of induction we can continue 
$\zeta_r((s_j);(\xi_j);(\gamma_j))$ meromorphically to region \eqref{a5}. 
Since $M$ is arbitrary, we can now conclude that
$\zeta_r((s_j);(\xi_j);(\gamma_j))$ can be continued meromorphically to the whole
space $\mathbb{C}^r$.
 
Next we examine the possible singularities on the right-hand side of \eqref{a4}.
By the assumption of induction, we see that the possible singularities of $Y(k)$ 
are
\begin{align}\label{a6}
s_j+\cdots+s_{r-2}+s_{r-1}+s_r+k=C(j,r-1)-\ell
\quad{\rm if}\quad\xi_j=1\quad(1\leqslant  j\leqslant  r-2)
\end{align}
and
\begin{align}\label{a7}
s_{r-1}+s_r+k=1\quad{\rm if}\quad\xi_{r-1}=1.
\end{align}
If $\xi_j\neq 1$ for all $j$ ($1\leqslant  j\leqslant  r-1$), then $Y(k)$ is entire.
The term $X$ appears only in case $\xi_r=1$, and in this case, $s_r=1$ is 
a possible singularity.   Moreover, by the assumption of induction we find the
following possible singularities of $X$:
\begin{align}\label{a8}                                                                
s_j+\cdots+s_{r-2}+s_{r-1}+s_r-1=C(j,r-1)-\ell                                          
\quad{\rm if}\quad\xi_j=1\quad(1\leqslant  j\leqslant  r-2,\ j=r)                                         
\end{align}
and
\begin{align}\label{a9}                                                                 
s_{r-1}+s_r-1=1\quad{\rm if}\quad\xi_{r-1}=1\ \text{and}\ \xi_r=1.                                           
\end{align}
If $\xi_j\neq 1$ for all $j$ ($1\leqslant  j\leqslant  r-1$), then $X$ is entire.
Since $k$ also runs over $\mathbb{N}_0$, renaming $k+\ell$ in \eqref{a6} and
$k$ in \eqref{a7} as $\ell$, we find that 
the above list of possible singularities can be rewritten as follows 
(where $\ell\in\mathbb{N}_0$).

$\bullet$ $s_j+\cdots+s_r=(C(j,r-1)+1)-\ell$ and $s_r=1$ if $\xi_j=1$ 
($1\leqslant  j\leqslant  r-2$) and $\xi_r=1$,

$\bullet$ $s_{r-1}+s_r=2-\ell$ and $s_r=1$ if $\xi_{r-1}=1$ and $\xi_r=1$ (given by \eqref{a7} and \eqref{a9}),

$\bullet$ $s_j+\cdots+s_r=C(j,r-1)-\ell$ if $\xi_j=1$ 
($1\leqslant  j\leqslant  r-2$) and $\xi_r\neq 1$, 

$\bullet$ $s_{r-1}+s_r=1-\ell$ if $\xi_{r-1}=1$ and $\xi_r\neq 1$.

\noindent
Since $C(j,r)=C(j,r-1)+1$ when $\xi_r=1$ and $C(j,r)=C(j,r-1)$ when  $\xi_r\neq 1$,
the factors $C(j,r-1)+1$ and $C(j,r-1)$ in the above list are all equal to
$C(j,r)$.
This completes the proof of the lemma, because we also notice that
$C(r-1,r)=2$ if $\xi_{r-1}=\xi_r=1$ and $C(r-1,r)=1$ if $\xi_{r-1}=1$ and 
$\xi_r\neq 1$.    
\end{proof}

Next we discuss whether the possible singularities listed in Lemma \ref{lmm-a} are 
indeed singularities, or not.    For this purpose, we first prepare the following

\begin{lmm}\label{lmm-a2}
Let $\xi\in\mathbb{C}$ with $|\xi|=1$.
If $\xi\neq\pm 1$, then $\phi(-k,\xi)\neq0$ for all $k\in\mathbb{N}_0$.
If $\xi=\pm 1$, then $\phi(-k,\xi)\neq0$ for all odd $k\in\mathbb{N}$ and $k=0$,
and $\phi(-k,\xi)=0$ for all even $k\in\mathbb{N}$.
\end{lmm}
\begin{proof}
If $\xi=\pm 1$, then we have
\begin{align*}
\phi(-k,1)&=\zeta(-k),  \\
\phi(-k,-1)&=(2^{1+k}-1)\zeta(-k),
\end{align*}
which reduces to the well-known cases.
In the following we assume that $\xi\neq\pm 1$.
Put $\xi=e^{2\pi i\theta}$ with $0<\theta<1$ and $\theta\neq1/2$.
It is known that
\begin{equation}\label{mosimosi}
  \frac{1}{1-\xi e^t}
=\sum_{k=0}^\infty \phi(-k,\xi)\frac{t^k}{k!}
\end{equation}
(cf. \cite[Section 1]{FKMT1}).
If $k=0$, then we have
\begin{equation*}
  \phi(0,\xi)=\frac{1}{1-\xi}\neq 0.
\end{equation*}
Assume $k\geqslant 1$.
For any sufficiently small $\varepsilon>0$, we have
  \begin{equation*}
    \begin{split}
      &\frac{\phi(-k,\xi)}{k!}
      =\frac{1}{2\pi i}\int_{|t|=\varepsilon}\frac{t^{-k-1}dt}{1-\xi e^t}
      \\
      &\quad=\frac{1}{2\pi i}\int_{|t|=\varepsilon}\frac{t^{-k-1}dt}{1-e^{t+2\pi i \theta}}
       =-\sum_{n\in\mathbb{Z}}\frac{1}{(2\pi i(n-\theta))^{k+1}},
    \end{split}
  \end{equation*}
where the last equality follows by counting residues at the poles 
$t=2\pi i(n-\theta)$.   Therefore it is sufficient to show that
\begin{equation}\label{a10}
  \sum_{n\in\mathbb{Z}}\frac{1}{(n-\theta)^{k+1}}\neq 0.
\end{equation}
If $k$ is odd, then the left-hand side is clearly positive.
If $k$ is even, then
\begin{equation*}
    \sum_{n\in\mathbb{Z}}\frac{1}{(n-\theta)^{k+1}}
=\sum_{n=0}^\infty
  \Biggl(\frac{1}{(n+1-\theta)^{k+1}}-
  \frac{1}{(n+\theta)^{k+1}}\Biggr)\neq 0,
\end{equation*}
because for all $n\geqslant 0$,
\begin{equation*}
  \frac{1}{(n+1-\theta)^{k+1}}-
  \frac{1}{(n+\theta)^{k+1}}
  \begin{cases}
    <0 \qquad & (0<\theta<1/2) \\
    >0 \qquad & (1/2<\theta<1).
  \end{cases}
\end{equation*}
The lemma is proved.
\end{proof}

Now our aim is to prove the following theorem, from which Theorem \ref{thm1}
immediately follows.

\begin{thm}\label{thm2}
Among the list of possible singularities of $\zeta_r((s_j);(\xi_j);(\gamma_j))$
given in Lemma \ref{lmm-a}, the ``true'' singularities are listed up as follows,
where $\ell\in\mathbb{N}_0$.

${\rm (I)}$ If $\xi_j=1$, then $s_j+\cdots+s_r=C(j,r)-\ell$ $\qquad$ 
{\rm(}$1\leqslant  j\leqslant  r-2${\rm)},

${\rm (II)}$ If $\xi_{r-1}=1$ and $\xi_r=1$, then 
$s_{r-1}+s_r=2,1,-2\ell$,

${\rm (III)}$ If $\xi_{r-1}=1$ and $\xi_r=-1$, then
$s_{r-1}+s_r=1,-2\ell$,

${\rm (IV)}$ If $\xi_{r-1}=1$ and $\xi_r\neq\pm 1$, then
$s_{r-1}+s_r=1-\ell$,

${\rm (V)}$ If $\xi_r=1$, then $s_r=1$.
\end{thm}

\begin{rem}
When $\xi_j=\gamma_j=1$ ($1\leqslant  j\leqslant  r$), this theorem recovers 
\cite[Theorem 1]{AET01}.
\end{rem}

\begin{proof}
The proof is by induction on $r$.   The case $r=1$ is obvious, so we assume
$r\geqslant  2$ and the theorem is true for $r-1$.

First we put $s_{r-1}+s_r=u$, and regard \eqref{a4} as a formula in variables
$s_1,\ldots,s_{r-2},u,s_r$.   This idea of ``changing variables'' is originally
due to Akiyama, Egami and Tanigawa \cite{AET01}.   We have
\begin{align*}
&X=\delta(r)\frac{\gamma_r^{-1}}{s_r-1}                                             
\zeta_{r-1}((s_1,\ldots,s_{r-2},u-1);                                        
(\xi_1,\ldots,\xi_{r-1});(\gamma_1,\ldots,\gamma_{r-1})),\\                        
&Y(k)=\binom{-s_r}{k}\zeta_{r-1}((s_1,\ldots,s_{r-2},u+k);        
(\xi_1,\ldots,\xi_{r-1});(\gamma_1,\ldots,\gamma_{r-1}))
\gamma_r^k \phi(-k,\xi_r).
\end{align*}
Consider $Y(k)$.   The singularities \eqref{a6} and \eqref{a7} are coming from 
the $\zeta_{r-1}$ factor.   These singularities do not be canceled by the factor
$\binom{-s_r}{k}$, because the $\zeta_{r-1}$ factor (after the above
``changing variables'') does not include the variable $s_r$.
Also, if $k^{\prime}\neq k$, then the singularities of $Y(k^{\prime})$ and of
$Y(k)$ do not cancel with each other, because $Y(k^{\prime})$ and $Y(k)$ is of
different order with respect to $s_r$.

When $\xi_r=1$, the term $X$ appears.   The possible singularities coming from $X$ 
are \eqref{a8}, \eqref{a9}, and $s_r=1$.   These singularities do not cancel with
each other.   Also, these singularities do not cancel
the singularities coming from $Y(k)$, which can be seen again by observing the order
with respect to $s_r$.

Therefore now we can say:

(i) The possible singularities of $Y(k)$ are ``true'' if they are ``true''
singularities of $\zeta_{r-1}$ and $\phi(-k,\xi_r)\neq 0$,

(ii) When $\xi_r=1$, the hyperplane $s_r=1$ is a ``true'' singularity, while the
other possible singularities of $X$ are ``true'' if they are ``true'' singularities of
$\zeta_{r-1}$.

Consider (i).  By the assumption of induction, the ``true'' singularities of
$$
\zeta_{r-1}((s_1,\ldots,s_{r-2},s_{r-1}+s_r+k);                                
(\xi_1,\ldots,\xi_{r-1});(\gamma_1,\ldots,\gamma_{r-1}))
$$
are

(i-1) $s_j+\cdots+s_r+k=C(j,r-1)-\ell$ if $\xi_j=1$ ($1\leqslant  j\leqslant  r-3$),

(i-2) $s_{r-2}+s_{r-1}+s_r+k=2,1,-2\ell$ if $\xi_{r-2}=1$, $\xi_{r-1}=1$,

(i-3) $s_{r-2}+s_{r-1}+s_r+k=1,-2\ell$ if $\xi_{r-2}=1$, $\xi_{r-1}=-1$,

(i-4) $s_{r-2}+s_{r-1}+s_r+k=1-\ell$ if $\xi_{r-2}=1$, $\xi_{r-1}\neq\pm 1$,

(i-5) $s_{r-1}+s_r+k=1$ if $\xi_{r-1}=1$.

\noindent
Here, by Lemma \ref{lmm-a2} we see that $k\in\mathbb{N}_0$ if $\xi_r\neq\pm 1$, 
while $k$ is
$0$ or odd positive integer if $\xi_r=\pm 1$.
Renaming $k+\ell$ in (i-1) as $\ell$, we can rewrite (i-1) as

(i-1') $s_j+\cdots+s_r=C(j,r-1)-\ell$ if $\xi_j=1$ ($1\leqslant  j\leqslant  r-3$).

\noindent
Next, the equality in (i-2) is $s_{r-2}+s_{r-1}+s_r=2-k,1-k,-2\ell-k$, and the 
right-hand side
exhausts all integers $\leqslant  2$ even in the case when $k$ is $0$ or odd positive 
integer.   Therefore (i-2) can be rewritten as

(i-2') $s_{r-2}+s_{r-1}+s_r=2-\ell$ if $\xi_{r-2}=1$, $\xi_{r-1}=1$.

\noindent
Similarly we rewrite (i-3) and (i-4) as

(i-3') $s_{r-2}+s_{r-1}+s_r=1-\ell$ if $\xi_{r-2}=1$, $\xi_{r-1}=-1$,

(i-4') $s_{r-2}+s_{r-1}+s_r=1-\ell$ if $\xi_{r-2}=1$, $\xi_{r-1}\neq\pm 1$.

\noindent
These (i-1')--(i-4') and (i-5) give the list of ``true'' singularities coming from 
the case (i).

Next consider (ii).   By the assumption of induction, the ``true'' singularities
of
$$
\delta(r)\frac{1}{s_r-1}                                               
\zeta_{r-1}((s_1,\ldots,s_{r-2},s_{r-1}+s_r-1);
(\xi_1,\ldots,\xi_{r-1});(\gamma_1,\ldots,\gamma_{r-1}))
$$
are

(ii-1) $s_j+\cdots+s_r-1=C(j,r-1)-\ell$ if $\xi_j=1$ ($1\leqslant  j\leqslant  r-3$), $\xi_r=1$,

(ii-2) $s_{r-2}+s_{r-1}+s_r-1=2,1,-2\ell$ if $\xi_{r-2}=1$, $\xi_{r-1}=1$, $\xi_r=1$,

(ii-3) $s_{r-2}+s_{r-1}+s_r-1=1,-2\ell$ if $\xi_{r-2}=1$, $\xi_{r-1}=-1$, $\xi_r=1$,

(ii-4) $s_{r-2}+s_{r-1}+s_r-1=1-\ell$ if $\xi_{r-2}=1$, $\xi_{r-1}\neq\pm 1$,
$\xi_r=1$,

(ii-5) $s_{r-1}+s_r-1=1$ if $\xi_{r-1}=1$, $\xi_r=1$,

\noindent and

(ii-6) $s_r=1$ if $\xi_r=1$.

\noindent 
The last (ii-6) is singularity (V) in the statement of Theorem \ref{thm2}.

From (i-1'), (ii-1) and the definition of $C(j,r)$ we obtain
$s_j+\cdots+s_r=C(j,r)-\ell$ if $\xi_j=1$ ($1\leqslant  j\leqslant  r-3$).   This gives
singularity (I) for $1\leqslant  j\leqslant  r-3$.

Consider the case $j=r-2$.   From (i-2') and (ii-2) we find that 
$s_{r-2}+s_{r-1}+s_r=3-\ell$ are singularities if $\xi_{r-2}=1$, $\xi_{r-1}=1$, 
$\xi_r=1$.
From (i-3'), (i-4'), (ii-3) and (ii-4) we find that
$s_{r-2}+s_{r-1}+s_r=2-\ell$ are singularities if $\xi_{r-2}=1$, $\xi_{r-1}\neq 1$,
$\xi_r=1$.
These observations and (i-2')--(i-4') imply that
$s_{r-2}+s_{r-1}+s_r=C(r-2,r)-\ell$ are singularities if $\xi_{r-2}=1$.
This is singularity (I) for $j=r-2$.

Finally, from (i-5) we obtain the singularities $s_{r-1}+s_r=1-\ell$ if
$\xi_{r-1}=1$, $\xi_r\neq\pm 1$, and $s_{r-1}+s_r=1,-2\ell$ if
$\xi_{r-1}=1$, $\xi_r=\pm 1$.   The former case gives singularity (IV).
The latter case, combined with (ii-5), gives singularities (II) and (III).
This completes the proof of the theorem.

\end{proof}

\begin{rem}
In the above proof, an important fact 
is that there are infinitely many $k\in \mathbb{N}$ with $\phi(-k,\xi)\not=0$. Actually, Lemma \ref{lmm-a2} ensures this fact. We can give another approach to ensure this fact. The number defined by
\begin{align}\label{kameyo}
H_k(\xi^{-1}):=(1-\xi)\phi(-k,\xi)\qquad (k\in \mathbb{N}_{0})
\end{align}
is called the $k$th Frobenius-Euler number studied by Frobenius in \cite{Fro}. 
He showed that, if $\xi$ is the primitive $c$th root of unity with $c>1$ and $p$ is an odd prime number with $p\nmid c$, then
$$H_k(\xi^{-1})\equiv \frac{1}{\xi^{-1}-1}\qquad (\text{mod\ }p)$$
for any $k\in \mathbb{N}_0$ with $k\equiv 1$ (mod $p-1$). Thus there are infinitely many $k\in \mathbb{N}$ with 
$$\phi(-k,\xi)=\frac{1}{1-\xi}H_k(\xi^{-1})\neq 0. $$
\end{rem}

\begin{rem}
It is desirable to generalize the results proved in this section to more general
multiple zeta-functions defined by \eqref{eq:L_series}, but it seems not easy,
because the argument based on Mellin-Barnes integrals will become more
complicated (see \cite{Mat03c}). 
\end{rem}


\let\bs\boldsymbol
\section{Desingularization of Multiple zeta-functions}\label{sec-3}
In this section,
we define desingularization of 
multiple zeta-functions of generalized Hurwitz-Lerch type \eqref{eq:L_series}, 
which includes those of generalized Euler-Zagier-Lerch type \eqref{Def-EZL-zeta}. 

Combining
the integral representation of gamma function
\begin{equation}
  \Gamma(s)=a^s\int_0^\infty e^{-ax}x^{s-1} dx
\end{equation}
for $a\in\mathbb{C}$ with $\Re a>0$, and
\begin{equation}
\label{eq:eyxi}
  \frac{1}{e^y-\xi}=\sum_{n=0}^\infty \xi^{n} e^{-(n+1)y}
\end{equation}
for $|\xi|\leqslant  1$ and $y>0$, 
the multiple zeta-function of generalized Hurwitz-Lerch type defined by \eqref{eq:L_series} is rewritten in the integral form as 
\begin{multline}
  \label{eq:int_rep1}
  \zeta_r((s_j);(\xi_k);(\gamma_{jk});(\beta_j))=
  \frac{1}{\Gamma(s_1)\cdots\Gamma(s_d)}
  \\
  \times
  \int_{[0,\infty)^d} 
  \frac{e^{(\gamma_{11}+\cdots+\gamma_{1r}-\beta_1)x_1}\cdots
    e^{(\gamma_{d1}+\cdots+\gamma_{dr}-\beta_d)x_d}x_1^{s_1-1}\cdots x_d^{s_d-1}}
  {(e^{x_1\gamma_{11}+\dots+x_d\gamma_{d1}}-\xi_1)\cdots(e^{x_1\gamma_{1r}+\dots+x_d\gamma_{dr}}-\xi_r)}
  dx_1\cdots dx_d
  \\
  =
  \frac{1}{\prod_{j=1}^d\Gamma(s_j)}
  \int_{[0,\infty)^d}
  \prod_{j=1}^d x_j^{s_j-1}\exp\Bigl((\sum_{k=1}^r \gamma_{jk} -\beta_j)x_j\Bigr)dx_j
  \prod_{k=1}^r\frac{1}{\exp\Bigl(\sum_{j=1}^d \gamma_{jk} x_j\Bigr)-\xi_k}.
\end{multline}
If $\xi_k\neq 1$ for all $k$, then, as was shown in \cite{Kom10}, it can be analytically continued to the whole space in $(s_j)$ as an entire function via the integral representation:
\begin{multline}
\label{eq:int_rep2}
  \zeta_r((s_j);(\xi_k);(\gamma_{jk});(\beta_j))=
  \frac{1}{\prod_{j=1}^d(e^{2\pi i s_j}-1)\Gamma(s_j)}
  \\
  \times
  \int_{\mathcal{C}^d}
  \prod_{j=1}^d x_j^{s_j-1}\exp\Bigl((\sum_{k=1}^r \gamma_{jk} -\beta_j)x_j\Bigr)dx_j
  \prod_{k=1}^r\frac{1}{\exp\Bigl(\sum_{j=1}^d \gamma_{jk} x_j\Bigr)-\xi_k},
\end{multline}
where $\mathcal{C}$ is the Hankel contour, that is, the path consisting of the positive real axis (top side), a circle around the origin of radius $\varepsilon$ (sufficiently small), and the positive real axis (bottom side).
The replacement of $[0,\infty)^d$ by the contour $\mathcal{C}^d$ can be checked directly (for the details, see \cite{Kom10}, where, more generally, the cases $\xi_j=1$ for some $j$ are treated).

Motivated as in \cite{FKMT1}, we introduce the notion of desingularization.

\begin{dfn} \label{def-MZF-2}
Let $\xi_k,\gamma_{jk},\beta_j\in \mathbb{C}$ with $|\xi_k|\leqslant  1$, 
$\Re \gamma_{jk}\geqslant  0$, $\Re \beta_j >0$, 
and for each $j$, at least one of $\Re\gamma_{jk}>0$.
Define the \textbf{desingularized multiple zeta-function},
which we also call the \textbf{desingularization of} 
$\zeta_r((s_j);(\xi_k);(\gamma_{jk});(\beta_j))$, by
  \begin{equation}\label{zeta-des-def-2-1}
    \begin{split}
      &\zeta_r^{\rm des}((s_j);(\xi_k);(\gamma_{jk});(\beta_j))
      \\
      &:=  
      \lim_{c\to1}
      \frac{1}{\prod_{k=1}^r(1-\delta(k)c)}
      \\
      &\times
      \frac{1}{\prod_{j=1}^d(e^{2\pi i s_j}-1)\Gamma(s_j)}
      \int_{\mathcal{C}^d}
      \prod_{j=1}^d x_j^{s_j-1}\exp\Bigl((\sum_{k=1}^r \gamma_{jk} -\beta_j)x_j\Bigr)dx_j
      \\
      &\times\prod_{k=1}^r\Biggl(\frac{1}{\exp\Bigl(\sum_{j=1}^d \gamma_{jk} x_j\Bigr)-\xi_k}
      -\delta(k)
      \frac{c}{\exp\Bigl(c\sum_{j=1}^d \gamma_{jk} x_j\Bigr)-1}
      \Biggr)
  \end{split}
\end{equation}
for $(s_j)\in \mathbb{C}^r$,
where the limit is taken for $c\in\mathbb{R}$ and
$\delta(k)$ is as in \eqref{a-delta}.
\end{dfn}

\begin{rem}\label{rem-entire}
If $\xi_k\neq 1$ for all $k$, then 
$\zeta_r((s_j);(\xi_k);(\gamma_{jk});(\beta_j))$ is already entire as we
mentioned above, so there is no need of desingularization.
In fact, since in this case $\delta(k)=0$ for all $k$, 
\eqref{zeta-des-def-2-1} coincides with \eqref{eq:int_rep2}.
\end{rem}

For $c\in\mathbb{R}$,
$y,\xi\in\mathbb{C}$, $\delta\in\{0,1\}$
with
$\delta=1$ if $\xi=1$, and $\delta=0$ otherwise,
let
\begin{equation*}                                                                       
  F_{c,\delta}(y,\xi)                                                                   
  =                                                                                     
  \begin{cases}                                                                         
    \displaystyle                                                                       
    \frac{1}{1-\delta c}\Bigl(\frac{1}{e^y-\xi}-\delta\frac{c}{(e^{cy}-1)}\Bigr)
\qquad&(c\neq 1), \\[4mm]  
    \displaystyle                                                                       
    \frac{1}{e^y-\xi}-\delta\frac{ye^y}{(e^y-1)^2}\qquad&(c=1),                         
  \end{cases}                                                                           
\end{equation*}
and further we write $F_{\delta}(y,\xi)=F_{1,\delta}(y,\xi)$.


\begin{thm}
\label{thm:ann}
  For $\xi_k,\gamma_{jk},\beta_j\in \mathbb{C}$ as in Definition \ref{def-MZF-2},
we have
  \begin{equation}\label{des-contour-2-2}
    \begin{split}
      &\zeta_r^{\rm des}((s_j);(\xi_k);(\gamma_{jk});(\beta_j))
      \\
      &=  
      \frac{1}{\prod_{j=1}^d(e^{2\pi i s_j}-1)\Gamma(s_j)}
      \int_{\mathcal{C}^d}
      \prod_{j=1}^d x_j^{s_j-1}\exp\Bigl((\sum_{k=1}^r \gamma_{jk} -\beta_j)x_j\Bigr)dx_j
      \\
      &\times\prod_{k=1}^r
      F_{\delta(k)}\left(\sum_{j=1}^d \gamma_{jk}x_j, \xi_k\right),
  \end{split}
\end{equation}
and is analytically continued to $\mathbb{C}^r$ as an entire function in $(s_j)$.
\end{thm}

Theorem \ref{thm:ann} can be shown in almost the same way as in 
\cite[Theorem 3.4]{FKMT1}.
We first use Lemma \ref{lm:conv} below in place of \cite[Lemma 3.6]{FKMT1} to find that
the limit and the multiple integrals on the right-hand side of
\eqref{zeta-des-def-2-1} can be interchanged. Then we use the following
Lemma \ref{lm:F} to obtain the assertion of Theorem \ref{thm:ann}.

\begin{lmm}
\label{lm:F}
  \begin{equation*}
    F_{1,\delta}(y,\xi)=
    \lim_{c\to1} F_{c,\delta}(y,\xi).
  \end{equation*}
\end{lmm}  
\begin{proof}
  If $\xi=\delta=1$, then 
\begin{equation}
  \begin{split}
    \lim_{c\to1}
    \frac{1}{1-\delta c}\Bigl(\frac{1}{e^y-\xi}-\delta\frac{c}{e^{cy}-1}\Bigr)
    &=\lim_{c\to1}
    \frac{1}{1-c}\Bigl(\frac{1}{e^{y}-1}-\frac{c}{e^{cy}-1}\Bigr)
    \\
    &=-\frac{1-e^{y}+ye^{y}}{(e^{y}-1)^2}
    \\
    &=\frac{1}{e^y-1}-\frac{ye^y}{(e^y-1)^2},
  \end{split}
\end{equation}
while if $\delta=0$ and $\xi\neq 1$, the assertion is obvious.
\end{proof}

Let $\mathcal{N}(\varepsilon)=\{z\in\mathbb{C}~|~|z|\leqslant  \varepsilon\}$ and 
$\mathcal{S}(\theta)=\{z\in\mathbb{C}~|~|\arg z|\leqslant  \theta\}$.
\begin{lmm} 
\label{lm:conv}
Let $0<\theta<\pi/2$.
Assume $|\xi|\leqslant 1$.
Then there exist $A>0$ and
sufficiently small $\varepsilon>0$ 
such that for all $c\in\mathbb{R}$
with sufficiently small $|1-c|$,
  \begin{equation}
    |F_{c,\delta}(y,\xi)|<A e^{-\Re y/2}
  \end{equation}
for any 
  $y\in \mathcal{N}(\varepsilon)\cup \mathcal{S}(\theta)$.
\end{lmm}
\begin{proof}
(1) Assume $\delta=0$ and $\xi\neq 1$.
Then
there exist $\varepsilon,C>0$ such that for all $y\in \mathcal{N}(\varepsilon)$, 
  \begin{equation*}
    |F_{c,\delta}(y,\xi)|=
    \Bigl|\frac{1}{e^y-\xi}\Bigr|<C.
  \end{equation*}
Further for
  $y\in \mathcal{S}(\theta)\setminus\mathcal{N}(\varepsilon)$, we have
  \begin{equation*}
    |F_{c,\delta}(y,\xi)|\leqslant 
    \frac{1}{|e^{y}|-1}=\frac{e^{-\Re y}}{1-e^{-\Re y}}\leqslant  C'e^{-\Re y}.
  \end{equation*}

(2) Assume $\delta=\xi=1$. Then this case reduces to \cite[Lemma 3.6]{FKMT1}.
\end{proof}

It is to be noted that the following continuity properties hold.

\begin{thm}
The desingularization $\zeta_r^{\rm des}((s_j);(\xi_k);(\gamma_{jk});(\beta_j))$
is continuous in both $(s_j)$ and $(\xi_k)$.   In particular, if $\xi_k\neq 1$ 
for all $k$,
then $\zeta_r((s_j);(\xi_k);(\gamma_{jk});(\beta_j))$ 
is continuous in both $(s_j)$ and $(\xi_k)$.
\end{thm}

\begin{proof}
The first statement follows easily from Lemma \ref{lm:conv} by using the dominated
convergence theorem.   The second statement is just a special case of the first
statement in view of Remark \ref{rem-entire}.
\end{proof}

Next we give a generating function of special values of
    $\zeta_r((s_j);(\xi_k);(\gamma_{jk});(\beta_j))$
at non-positive integers. 
Write the Taylor expansion of $F_\delta(y,\xi)$ with respect to $y$ as
\begin{equation}\label{F-taylor}
    F_\delta(y,\xi)
    =\frac{1}{e^y-\xi}-\delta\frac{ye^y}{(e^y-1)^2}
    =\sum_{n=0}^\infty F_\delta^n(\xi)\frac{y^n}{n!}.
\end{equation}
Then
\begin{equation}\label{F-taylor-coeff}
  F_\delta^n(\xi)=
  \begin{cases}
      \displaystyle
      B_{n+1} \qquad &(\xi=1, \delta=1),\\[4mm]
      \displaystyle
      \dfrac{H_{n}(\xi)}{1-\xi} \qquad &(\xi\neq 1,\delta=0), 
  \end{cases}
\end{equation}
where $B_{n+1}$ denotes the $(n+1)$-th Bernoulli number.
The first formula of \eqref{F-taylor-coeff} can be shown by differentiating the
definition of Bernoulli numbers
\begin{equation}\label{Ber-def}
\frac{y}{e^y-1}=\sum_{n=0}^{\infty}B_n \frac{y^n}{n!},
\end{equation}
while the second formula follows from \eqref{mosimosi} and \eqref{kameyo}.

\begin{thm} Let $\lambda_1,\ldots,\lambda_d\in\mathbb{N}_0$. Then we have
\begin{multline} 
\label{eq:formula_spv}
    \zeta_r^{\rm des}((-\lambda_j);(\xi_k);(\gamma_{jk});(\beta_j))=
    \\
    \prod_{j=1}^d(-1)^{\lambda_j}\lambda_j!
\sum_{\substack{m_j+\nu_{j1}+\cdots+\nu_{jr}=\lambda_j\\
    (1\leqslant  j\leqslant  d)}}
\Bigl(\prod_{k=1}^r F_{\delta(k)}^{\nu_{1k}+\cdots+\nu_{dk}}(\xi_k)\Bigr)
\\
\times
\Bigl(\prod_{j=1}^d \frac{(\sum_{k=1}^r \gamma_{jk} -\beta_j)^{m_j}}{m_j!}\Bigr)
\Bigl(\prod_{j=1}^d\prod_{k=1}^r
\frac{\gamma_{jk}^{\nu_{jk}}}{\nu_{jk}!}\Bigr).
\end{multline}
\end{thm}

\begin{proof}
  Let $D_j=\sum_{k=1}^r \gamma_{jk} -\beta_j$.  
  It is sufficient to calculate the
  Taylor expansion with respect to $x_j$'s of the integrand on the right-hand side
of \eqref{des-contour-2-2}.   Using \eqref{F-taylor} we have
\begin{multline}
  \prod_{j=1}^d \exp\Bigl((\sum_{k=1}^r \gamma_{jk} -\beta_j)x_j\Bigr)\prod_{k=1}^r 
F_{\delta(k)}\Bigl(\sum_{j=1}^d \gamma_{jk} x_j,\xi_k\Bigr)
  \\
  \begin{aligned}
    &=\sum_{m_{1},\ldots,m_{d}=0}^\infty 
    \left(\prod_{j=1}^d \frac{D_j^{m_j}}{m_j!}x_j^{m_j}\right)
    \sum_{n_{1},\ldots,n_{r}=0}^\infty
    \prod_{k=1}^r
    \frac{F_{\delta(k)}^{n_k}(\xi_k)}{n_k!}\Bigl(\sum_{j=1}^d \gamma_{jk} 
x_{j}\Bigr)^{n_k}
    \\
    &=\sum_{m_{1},\ldots,m_{d}=0}^\infty 
    \left(\prod_{j=1}^d \frac{D_j^{m_j}}{m_j!}x_j^{m_j}\right) 
    \sum_{n_{1},\ldots,n_{r}=0}^\infty
    \prod_{k=1}^r
    \frac{F_{\delta(k)}^{n_k}(\xi_k)}{n_k!}
    \\
    &\qquad\qquad\qquad\qquad\times 
    \sum_{\nu_{1k}+\cdots+\nu_{dk}=n_k}
    \binom{n_k}{\nu_{1k},\ldots,\nu_{dk}}\prod_{j=1}^d\gamma_{jk}^{\nu_{jk}}
    x_j^{\nu_{jk}}
    \\
    &=\sum_{m_{1},\ldots,m_{d}=0}^\infty 
    \left(\prod_{j=1}^d \frac{D_j^{m_j}}{m_j!}x_j^{m_j}\right) 
    \sum_{n_{1},\ldots,n_{r}=0}^\infty
    \prod_{k=1}^r    
    \sum_{\nu_{1k}+\cdots+\nu_{dk}=n_k}
    \frac{F_{\delta(k)}^{n_k}(\xi_k)}{\nu_{1k}!\cdots \nu_{dk}!}
    \prod_{j=1}^d \gamma_{jk}^{\nu_{jk}}x_j^{\nu_{jk}}
    \\
    &=\sum_{m_{1},\ldots,m_{d}=0}^\infty 
    \sum_{n_{1},\ldots,n_{r}=0}^\infty
    \sum_{\substack{\nu_{1k}+\cdots+\nu_{dk}=n_k\\
        (1\leqslant  k\leqslant  r)}}
    \Bigl(\prod_{k=1}^r F_{\delta(k)}^{n_k}(\xi_k)\Bigr)
    \Bigl(\prod_{j=1}^d \frac{D_j^{m_j}}{m_j!}\Bigr)
    \Bigl(\prod_{j=1}^d\prod_{k=1}^r
    \frac{\gamma_{jk}^{\nu_{jk}}}{\nu_{jk}!}\Bigr)
    \\
    &\qquad\qquad\qquad\qquad\times
    \prod_{j=1}^d
    x_j^{m_j+\nu_{j1}+\cdots+\nu_{jr}},
  \end{aligned}
\end{multline}
which gives the formula \eqref{eq:formula_spv}.

\end{proof}

\begin{rem}
  Since 
$D_j=0$ for all $j=1,\ldots,d$
in the case of 
multiple zeta-functions of generalized Euler-Zagier-Lerch
type, 
only terms with $m_j=0$ with $j=1,\ldots,d$
contributes to the sum in the formula \eqref{eq:formula_spv},
which recovers \cite[Theorem 3.7]{FKMT1}.
\end{rem}

Lastly, we give a formula which expresses the desingularized zeta-function as a linear combination of ordinary zeta-functions of the same type, which is a generalization of
\cite[Theorem 3.8]{FKMT1}.
To this end, we prepare some notation and assume the following condition: There exists a set of constants $c_{mj}$ ($1\leqslant  k,m\leqslant  r$) such that 
\begin{equation}
\label{eq:assumption}
  \sum_{j=1}^d c_{mj}\gamma_{jk}=\delta_{mk}
\end{equation}
for all $k,m$,
where $\delta_{mk}$ is the Kronecker delta.
Under this assumption,
 for indeterminates $\bs{u}=(u_j),\bs{v}=(v_j)$ ($j=1,\ldots,d$),
we define the generating function
\begin{equation}
G(\bs{u},\bs{v})=
\prod_{k=1}^r\Bigl\{1-\delta(k)\Bigl(
1+\sum_{j=1}^d c_{kj} (v_j^{-1}-\beta_j)
\Bigr)
\Bigl(\sum_{j=1}^d \gamma_{jk}u_jv_j\Bigr)\Bigr\},
\end{equation}
and also define constants $\alpha_{\bs{l},\bs{m}}$ as the coefficients of the expansion
\begin{equation}
  G(\bs{u},\bs{v})=\sum_{\bs{l},\bs{m}} \alpha_{\bs{l},\bs{m}}\prod_{j=1}^d u_j^{l_j}v_j^{m_j} \qquad \bs{l}=(l_1,\ldots,l_d), \quad \bs{m}=(m_1,\ldots,m_d).
\end{equation}
We define the Pochhammer symbol $(s)_k=s(s+1)\cdots(s+k-1)$ as usual.
Then we have the following theorem, which is a generalization of
\cite[Theorem 3.8]{FKMT1}.
\begin{thm}
\label{thm:des}
Under the assumption \eqref{eq:assumption}, we have
\begin{multline}
    \zeta_r^{\rm des}((s_j);(\xi_k);(\gamma_{jk});(\beta_j))
\\
=\sum_{\bs{l},\bs{m}} \alpha_{\bs{l},\bs{m}}\Bigl(\prod_{j=1}^d(s_j)_{l_j}\Bigr)
\zeta_r((s_j+m_j);(\xi_k);(\gamma_{jk});(\beta_j)).
\end{multline}
\end{thm}

\begin{proof}
First note that it is sufficient to show the statement
with sufficiently large $\Re s_j$ due to the analytic continuation.
Then we can write
\begin{equation}\label{des-eq-lim}
\zeta_r^{\rm des}((s_j);(\xi_k);(\gamma_{jk});(\beta_j))
=\lim_{c\to 1}\frac{I_c((s_j);(\xi_k);(\gamma_{jk});(\beta_j))}
{\prod_{k=1}^r(1-\delta(k)c)},
\end{equation}
where
\begin{multline}
  I_{c}((s_j);(\xi_k);(\gamma_{jk});(\beta_j)):=
  \frac{1}{\prod_{j=1}^d\Gamma(s_j)}
  \int_{[0,\infty)^d}
  \prod_{j=1}^d x_j^{s_j-1}\exp\Bigl((\sum_{k=1}^r \gamma_{jk} -\beta_j)x_j\Bigr)dx_j
  \\
  \times\prod_{k=1}^r\Biggl(\frac{1}{\exp\Bigl(\sum_{j=1}^d \gamma_{jk} x_j\Bigr)-\xi_k}
  -\delta(k)
  \frac{c}{\exp\Bigl(c\sum_{j=1}^d \gamma_{jk} x_j\Bigr)-1}
  \Biggr).
\end{multline}

We obtain
\begin{multline}
  \lim_{c\to1}\frac{I_{c}((s_j);(\xi_k);(\gamma_{jk});(\beta_j))}
{\prod_{k=1}^r(1-\delta(k)c)}
  \\
  \begin{aligned}
    &=
    \lim_{c\to1}
    \frac{1}{\prod_{j=1}^d\Gamma(s_j)}
    \int_{[0,\infty)^d}
    \prod_{j=1}^d x_j^{s_j-1}\exp\Bigl((\sum_{k=1}^r \gamma_{jk} -\beta_j)x_j\Bigr)dx_j
    \\
    &
    \qquad
    \times
    \prod_{k=1}^r
    F_{c,\delta(k)}\Bigl(\sum_{j=1}^d \gamma_{jk} x_j,\xi_k\Bigr)
    \\
    &=
    \frac{1}{\prod_{j=1}^d\Gamma(s_j)}
    \int_{[0,\infty)^d}
    \prod_{j=1}^d x_j^{s_j-1}\exp\Bigl((\sum_{k=1}^r \gamma_{jk} -\beta_j)x_j\Bigr)dx_j
    \\
    &
    \qquad
    \times
    \prod_{k=1}^r
    F_{\delta(k)}\Bigl(\sum_{j=1}^d \gamma_{jk} x_j,\xi_k\Bigr).
  \end{aligned}
\end{multline}
Since for $|\xi|\leqslant  1$ and $y>0$, equation \eqref{eq:eyxi} and
\begin{equation}
  \frac{e^y}{(e^y-1)^2}=\sum_{n=0}^\infty (n+1) e^{-(n+1)y}
\end{equation}
holds.
Using these formulas, for any $K\subset\{1,\ldots,r\}$ we have
\begin{multline}
\label{eq:c1}
  \int_{[0,\infty)^d}
  \prod_{j=1}^d x_j^{s_j-1}\exp\Bigl((\sum_{k=1}^r \gamma_{jk} -\beta_j)x_j\Bigr)dx_j
  \\
  \times
  \prod_{k\notin K}\frac{1}{\exp\Bigl(\sum_{j=1}^d \gamma_{jk}x_j\Bigr)-\xi_k}
  \prod_{k\in K}\delta(k)\frac{\Bigl(\sum_{j=1}^d \gamma_{jk}x_j\Bigr)\exp\Bigl(\sum_{j=1}^d \gamma_{jk}x_j\Bigr)}{\Big(\exp\Bigl(\sum_{j=1}^d \gamma_{jk}x_j\Bigr)-1\Bigr)^2}
  \\
  \begin{aligned}
    &=
    \int_{[0,\infty)^d}
    \prod_{j=1}^d x_j^{s_j-1}\exp\Bigl((\sum_{k=1}^r \gamma_{jk} -\beta_j)x_j\Bigr)dx_j
    \prod_{k\in K}\delta(k)\Bigl(\sum_{j=1}^d \gamma_{jk}x_j\Bigr)
    \\
    &\qquad\times
    \prod_{k\notin K}\Biggl(\sum_{h_k=0}^\infty\xi_k^{h_k}\exp\Bigl(-(h_k+1)\sum_{j=1}^d \gamma_{jk}x_j\Bigr)\Biggr)
    \\
    &\qquad\times
    \prod_{k\in K}\Biggl(\sum_{h_k=0}^\infty (h_k+1)\exp\Bigl(-(h_k+1)\sum_{j=1}^d \gamma_{jk}x_j\Bigr)\Biggr)
    \\
    &=
    \sum_{\substack{h_k\geqslant 0\\1\leqslant  k\leqslant  r}}
    \Bigl(\prod_{k\in K}(h_k+1)\Bigr)
    \int_{[0,\infty)^d}
    \prod_{j=1}^d x_j^{s_j-1}\exp\Bigl(-(\sum_{k=1}^r \gamma_{jk} h_k +\beta_j)x_j\Bigr)dx_j
    \\
    &\qquad\times\prod_{k\notin K} \xi_k^{h_k}
    \prod_{k\in K}\delta(k)\Bigl(\sum_{j=1}^d \gamma_{jk}x_j\Bigr).
  \end{aligned}
\end{multline}
(When $K=\emptyset$, the empty product is to be regarded as $1$.)
Since $\delta(k)=\delta(k)\xi_k^{h_k}$, we have
\begin{equation}
  \prod_{k\in K}\delta(k)
  \prod_{k\notin K}\xi_k^{h_k}=
  \prod_{k\in K}\delta(k)
  \prod_{k=1}^r\xi_k^{h_k}.
\end{equation}
Also, since we assume \eqref{eq:assumption}, we can write
\begin{equation}
    \prod_{k\in K}(h_k+1)=  
    \prod_{l\in K}\Bigl( 
    \sum_{j=1}^d c_{lj}(\beta_j+\sum_{k=1}^r \gamma_{jk}h_k-\beta_j)
    +1\Bigr).
\end{equation}
Therefore, introducing constants $B_{K,\bs{l}}$ with 
$\bs{l}=(l_1,\ldots,l_d)\in\mathbb{N}_0^d$ as the coefficients of the expansion
\begin{equation}
  \label{eq:gen_B}
  \prod_{k\in K}\delta(k)\Bigl(\sum_{j=1}^d \gamma_{jk}x_j\Bigr)=\sum_{\bs{l}} B_{K,\bs{l}} \prod_{j=1}^d x_j^{l_j},
\end{equation}
we find that \eqref{eq:c1} is equal to
\begin{align}
\label{eq:c2}
    &
    \sum_{\bs{l}}B_{K,\bs{l}}
    \sum_{\substack{h_k\geqslant 0\\1\leqslant  k\leqslant  r}}
    \Bigl(\prod_{m\in K}\Bigl(
    \sum_{j=1}^d c_{mj}(\beta_j+\sum_{k=1}^r \gamma_{jk}h_k-\beta_j)
    +1\Bigr)
    \Bigr)   \prod_{k=1}^r \xi_k^{h_k}
    \\
    &\qquad\times
    \int_{[0,\infty)^d}
    \prod_{j=1}^d x_j^{s_j+l_j-1}\exp\Bigl(-(\sum_{k=1}^r \gamma_{jk} h_k +\beta_j)x_j\Bigr)dx_j
    \notag\\
    &=
    \sum_{\bs{l}}B_{K,\bs{l}}\sum_{\substack{h_k\geqslant 0\\1\leqslant  k\leqslant  r}}
    \Bigl(\prod_{m\in K}\Bigl(  
    \sum_{j=1}^d c_{mj}(\beta_j+\sum_{k=1}^r \gamma_{jk}h_k)+c_{m0}
    \Bigr)
    \Bigr) \prod_{k=1}^r \xi_k^{h_k}
    \notag\\
    &\qquad\times
    \prod_{j=1}^d\Gamma(s_j+l_j)
    \frac{1}{\Bigl(\beta_j+\sum_{k=1}^r \gamma_{jk}h_k\Bigr)^{s_j+l_j}},\notag
\end{align}
where $c_{m0}:=1-\sum_{j=1}^d c_{mj}\beta_j$. 
Consider the factor
$$
\prod_{m\in K}\Bigl(                                                            
    \sum_{j=1}^d c_{mj}(\beta_j+\sum_{k=1}^r \gamma_{jk}h_k)+c_{m0}                     
    \Bigr)\;(=: Q, \;{\rm say})
$$
on the right-hand side of \eqref{eq:c2}.
Putting
\begin{align}                                                                           
\alpha_j:=\left\{                                                                       
\begin{array}{ll}                                                                       
    \beta_j+\sum_{k=1}^r \gamma_{jk}h_k & (1\leqslant  j\leqslant  d),\\                
    1 & (j=0),                                                                          
\end{array}\right.                                                                      
\end{align}
we find that 
$$                                                                                      
Q=\prod_{m\in K}\sum_{j=0}^d c_{mj}\alpha_j                                              
=\sum_{\substack{0\leqslant j_m\leqslant d\\m\in K}} \Bigl(\prod_{m\in K}c_{mj_m}\Bigr) 
\Bigl(\prod_{m\in K}\alpha_{j_m}\Bigr).
$$
For each $\{j_m\;|\;m\in K\}$, define
$$
J(j)=J(j;\{j_m\}):=|\{m\in K\;|\;j_m=j\}|=\sum_{m\in K}\delta_{j,j_m}\qquad
(1\leqslant  j\leqslant  d).
$$
Then we see that
$$                                                                                      
\prod_{m\in K}\alpha_{j_m}=\prod_{j=1}^d \alpha_j^{J(j)}. 
$$
Therefore we obtain
\begin{align}\label{tochuu}
Q=\sum_{\substack{0\leqslant j_m\leqslant d\\m\in K}}
    \Bigl(\prod_{m\in K}c_{mj_m}\Bigr)    
    \prod_{j=1}^d(\beta_j+\sum_{k=1}^r \gamma_{jk}h_k)^{J(j)}.
\end{align}

Using \eqref{tochuu} we find that
\eqref{eq:c2} can be rewritten as
\begin{align}
    &
    \sum_{\bs{l}}B_{K,\bs{l}}
    \Bigl(\prod_{j=1}^d\Gamma(s_j+l_j)\Bigr)
    \notag\\
    &\quad\times
    \sum_{\substack{h_k\geqslant 0\\1\leqslant  k\leqslant  r}}
    \Biggl(
    \sum_{\substack{0\leqslant j_m\leqslant d\\m\in K}}
    \Bigl(
    \prod_{m\in K}
    c_{mj_m}
    \Bigr)
    \frac{\prod_{k=1}^r \xi_k^{h_k}}{\prod_{j=1}^d
    \Bigl(\beta_j+\sum_{k=1}^r \gamma_{jk}h_k\Bigr)^{s_j+l_j-J(j)}}\Biggr)
    \notag\\
    &=
    \sum_{\bs{l}}B_{K,\bs{l}}
    \Bigl(\prod_{j=1}^d\Gamma(s_j+l_j)\Bigr)
    \sum_{\substack{0\leqslant j_m\leqslant d\\m\in K}}
    \Bigl(\prod_{m\in K}c_{mj_m}\Bigr)
\zeta_r\Bigl((s_j+l_j-J(j));(\xi_k);(\gamma_{jk});(\beta_j)\Bigr).\notag
\end{align}
Therefore from \eqref{des-eq-lim} we obtain

\begin{multline}
\zeta_r^{\rm des}((s_j);(\xi_k);(\gamma_{jk});(\beta_j))
    \\
    \begin{aligned}
      &=
    \sum_{K\subset\{1,\ldots,r\}}(-1)^{|K|}
    \sum_{\bs{l}}B_{K,\bs{l}}
    \Bigl(\prod_{j=1}^d(s_j)_{l_j}
    \Bigr)
    \\
    &\qquad\times
    \sum_{\substack{0\leqslant j_m\leqslant d\\m\in K}}
    \Bigl(\prod_{m\in K}c_{mj_m}\Bigr)
    \zeta_r\Bigl((s_j+l_j-J(j));(\xi_k);(\gamma_{jk});(\beta_j)\Bigr).
  \end{aligned}
\end{multline}
Put
\begin{equation}                                                                        
  H(\bs{u},\bs{v}):= 
  \sum_{K\subset\{1,\ldots,r\}}(-1)^{|K|}                                               
  \sum_{\bs{l}}B_{K,\bs{l}}                                                             
  \sum_{\substack{0\leqslant j_m\leqslant d\\m\in K}}
  \Bigl(\prod_{m\in K}c_{mj_m}\Bigr)                                                    
  \prod_{j=1}^d                                                                         
  u_j^{l_j}v_j^{l_j-J(j)}.                                       
\end{equation}
Our last task is to show that
\begin{equation}\label{G=H}
  G(\bs{u},\bs{v})=H(\bs{u},\bs{v}).
\end{equation}
From
\eqref{eq:gen_B}, we have
\begin{align}\label{H-u-v}
    H(\bs{u},\bs{v})
    &=
    \sum_{K\subset\{1,\ldots,r\}}(-1)^{|K|}
    \sum_{\substack{0\leqslant j_m\leqslant d\\m\in K}}
    \Bigl(\prod_{m\in K}c_{mj_m}\Bigr)
    \Bigl(\prod_{j=1}^d
    v_j^{-J(j)}\Bigr)
    \sum_{\bs{l}}B_{K,\bs{l}}
    \prod_{j=1}^d
    u_j^{l_j}v_j^{l_j}
    \\
    &=
    \sum_{K\subset\{1,\ldots,r\}}(-1)^{|K|}
    \sum_{\substack{0\leqslant j_m\leqslant d\\m\in K}} 
    \Bigl(\prod_{m\in K}c_{mj_m}\Bigr)
    \Bigl(\prod_{m\in K}\prod_{j=1}^d
    v_j^{-\delta_{j,j_m}}\Bigr)
    \prod_{k\in K}\delta(k)\Bigl(\sum_{j=1}^d \gamma_{jk}u_jv_j\Bigr).
    \notag
\end{align}
Since we see that
\begin{align*}
\prod_{j=1}^d v_j^{-\delta_{j,j_m}}=\left\{
   \begin{array}{ll}
      v_{j_m}^{-1}  &  (j_m\geqslant  1),\\
      1             &  (j_m=0),
   \end{array}\right.
\end{align*}
under the convention $v_0=1$, we find that the right-hand side of \eqref{H-u-v} is equal to 
\begin{align}
    &
    \sum_{K\subset\{1,\ldots,r\}}(-1)^{|K|}                                           
    \sum_{\substack{0\leqslant j_m\leqslant d\\m\in K}}
    \Bigl(\prod_{m\in K}c_{mj_m}v_{j_m}^{-1}\Bigr)
    \prod_{k\in K}\delta(k)\Bigl(\sum_{j=1}^d \gamma_{jk}u_jv_j\Bigr)
    \notag
    \\
    &=
    \sum_{K\subset\{1,\ldots,r\}}(-1)^{|K|}
    \Bigl(\prod_{m\in K}\Bigl(  
    \sum_{j=1}^d c_{mj} v_j^{-1}+1-\sum_{j=1}^d c_{mj}\beta_j
    \Bigr)
    \Bigr)
    \notag\\
    &\qquad\times
    \prod_{k\in K}\delta(k)\Bigl(\sum_{j=1}^d \gamma_{jk}u_jv_j\Bigr)
    \notag\\
    &=
    \sum_{K\subset\{1,\ldots,r\}}(-1)^{|K|}
    \Bigl\{\prod_{k\in K}\delta(k)\Bigl(  
    1+\sum_{j=1}^d c_{kj} (v_j^{-1}-\beta_j)
    \Bigr)
    \Bigl(\sum_{j=1}^d \gamma_{jk}u_jv_j\Bigr)\Bigr\}
    \notag\\
    &=
    \prod_{k=1}^r\Bigl\{1-\delta(k)\Bigl( 
    1+\sum_{j=1}^d c_{kj} (v_j^{-1}-\beta_j)
    \Bigr)
    \Bigl(\sum_{j=1}^d \gamma_{jk}u_jv_j\Bigr)\Bigr\}=G(\bs{u},\bs{v}),\notag
\end{align}
hence \eqref{G=H}.   Therefore, regarding $(s_j)_{l_j}$ and
$\zeta_r\Bigl((s_j+l_j-J(j));(\xi_k);(\gamma_{jk});(\beta_j)\Bigr)$ as
indeterminates $u_j^{l_j}$ and $v_j^{l_j-J(j)}$, respectively, we arrive at the
assertion of the theorem.
\end{proof}

\section{Examples of desingularization}\label{sec-4}

Our Theorem \ref{thm:des} in the preceding section requires the assumption
\eqref{eq:assumption}.   In this section we see how this assumption is satisfied 
in examples.

\begin{exa}
  In the case of the triple zeta-function of generalized Euler-Zagier-Lerch type ($d=r=3$), we have
  \begin{gather}
    (\xi_k)=
    \begin{pmatrix}
      1 & 1& 1
    \end{pmatrix},\qquad
    (\beta_j)=
    \begin{pmatrix}
      \gamma_1\; & \gamma_1+\gamma_2\; & \gamma_1+\gamma_2+\gamma_3
    \end{pmatrix},
    \\
    (c_{mj})=
    \begin{pmatrix}
      \gamma_1^{-1} & 0 & 0 \\
      -\gamma_2^{-1} & \gamma_2^{-1} & 0 \\
      0 & -\gamma_3^{-1} & \gamma_3^{-1} 
    \end{pmatrix},
    \qquad
    (\gamma_{jk})=
    \begin{pmatrix}
      \gamma_1 & 0 & 0 \\
      \gamma_1 & \gamma_2 & 0 \\
      \gamma_1 & \gamma_2 & \gamma_3 
    \end{pmatrix}.
  \end{gather}
The generating function constructed by using these data
coincides with $G(\bs{u},\bs{v})$ in \cite[Example 4.4]{FKMT1}.
\end{exa}


\begin{exa}
\label{exa:MT2}
Consider the case of the
  Mordell-Tornheim double zeta-function, which is defined by the double series
  \begin{equation}\label{Def-MT}
    \zeta_{MT,2}(s_1,s_2,s_3)=\sum_{m_1=1}^\infty\sum_{m_2=1}^\infty
    \frac{1}{m_1^{s_1}m_2^{s_2}(m_1+m_2)^{s_3}}
  \end{equation}
(cf. \cite{Mat02} \cite{Mat03c}), corresponding to $d=3$ and $r=2$.
In this case, 
constants $c_{mj}$ are not uniquely determined. For any $a,b\in\mathbb{C}$, we have
  \begin{gather}
    (\xi_k)=
    \begin{pmatrix}
      1 & 1
    \end{pmatrix},\qquad
    (\beta_j)=
    \begin{pmatrix}
      1 & 1 & 2
    \end{pmatrix},
    \\
    (c_{mj})=
    \begin{pmatrix}
      a+1 & a & -a \\
      b & b+1 & -b  
    \end{pmatrix},
    \qquad
    (\gamma_{jk})=
    \begin{pmatrix}
      1 & 0 \\
      0 & 1 \\
      1 & 1 
    \end{pmatrix}.
  \end{gather}
Therefore we have
\begin{align*}
    G(\bs{u},\bs{v})
    &=(1-v_1^{-1}(u_1v_1+u_3v_3))(1-v_2^{-1}(u_2v_2+u_3v_3))\\
    &\quad-(1-v_2^{-1}(u_2v_2+u_3v_3))(v_1^{-1}+v_2^{-1}-v_3^{-1})(u_1v_1+u_3v_3)a\\
    &\quad-(1-v_1^{-1}(u_1v_1+u_3v_3))(v_1^{-1}+v_2^{-1}-v_3^{-1})(u_2v_2+u_3v_3)b\\
    &\quad+(v_1^{-1}+v_2^{-1}-v_3^{-1})^2(u_1v_1+u_3v_3)(u_2v_2+u_3v_3)ab
    \\
    &=(u_1-1)(u_2-1)+u_3(u_1-1)v_2^{-1}v_3+u_3(u_2-1)v_1^{-1}v_3+u_3^2v_1^{-1}v_2^{-1}v_3^2\\
    &\quad+\Bigl\{(u_2-1)(u_1-u_3)-u_3(1-u_1-u_2+u_3)v_2^{-1}v_3+u_3^2v_2^{-2}v_3^2\\
    &\qquad+u_1(u_2-u_3-1)v_1v_2^{-1}-u_1(u_2-1)v_1v_3^{-1}+u_1u_3v_1v_2^{-2}v_3\\
    &\qquad+(u_2-1)u_3v_1^{-1}v_3+u_3^2v_1^{-1}v_2^{-1}v_3^{2}\Bigr\}a
    \\
    &\quad+\Bigl\{(u_1-1)(u_2-u_3)-u_3(1-u_1-u_2+u_3)v_1^{-1}v_3+u_3^2v_1^{-2}v_3^2\\
    &\qquad+u_2(u_1-u_3-1)v_1^{-1}v_2-u_2(u_1-1)v_2v_3^{-1}+u_2u_3v_1^{-2}v_2v_3\\
    &\qquad+(u_1-1)u_3v_2^{-1}v_3+u_3^2v_1^{-1}v_2^{-1}v_3^{2}\Bigr\}b
    \\
    &\quad+\Bigl\{u_3^2-2 u_1 u_3-2 u_2 u_3+2 u_1 u_2
    \\
    &\qquad
    +u_3 (u_1+2 u_2-2 u_3) v_1^{-1} v_3 
    +u_3 (2 u_1+u_2-2 u_3) v_2^{-1} v_3
    \\
    &\qquad
    +u_1 u_3 v_1 v_2^{-2} v_3 
    +u_2 u_3 v_1^{-2} v_2 v_3 
    \\
    &\qquad
    +u_1 (u_2-2 u_3) v_1 v_2^{-1}
    +u_2 (u_1-2 u_3) v_1^{-1} v_2 
    \\
    &\qquad 
    -u_1 (2 u_2-u_3) v_1 v_3^{-1}
    -u_2 (2 u_1-u_3) v_2 v_3^{-1}
    \\
    &\qquad
    +u_3^2 v_1^{-2}v_3^2 
    +u_3^2 v_2^{-2} v_3^2 
    \\
    &\qquad
    +u_1 u_2 v_1 v_2 v_3^{-2}
    +2 u_3^2 v_1^{-1} v_2^{-1} v_3^2
    \Bigr\}ab.
\end{align*}
From the constant part of this expression with respect to $a$ and $b$, we obtain
the following identity, which is an example of Theorem \ref{thm:des}. 
\begin{multline}
\zeta_{MT,2}^{\rm des}(s_1,s_2,s_3)=
  (s_1-1)(s_2-1)\zeta_{MT,2}(s_1,s_2,s_3)+s_3(s_1-1)\zeta_{MT,2}(s_1,s_2-1,s_3+1)
  \\
  +s_3(s_2-1)\zeta_{MT,2}(s_1-1,s_2,s_3+1)
  +s_3(s_3+1)\zeta_{MT,2}(s_1-1,s_2-1,s_3+2).
\end{multline}
On the other hand, coefficients of $a$, $b$, and $ab$ give rise to the following 
identities\footnote{Here, the second term of \eqref{eq:triv1} is not $s_3(1-s_1-s_2+s_3)$, but
$s_3(2-s_1-s_2+s_3)$, because the factor corresponding to 
$u_3(1+u_3)=u_3+u_3^2$ is not $s_3(1+s_3)$, but 
$s_3+s_3(s_3+1)=s_3(2+s_3)$.
}:
\begin{gather}
  \label{eq:triv1}
  \begin{aligned}
    &(s_2-1)(s_1-s_3)\zeta_{MT,2}(s_1,s_2,s_3)
    -s_3(2-s_1-s_2+s_3)\zeta_{MT,2}(s_1,s_2-1,s_3+1)
    \\
    &\qquad
    +s_3(s_3+1)\zeta_{MT,2}(s_1,s_2-2,s_3+2)
    +s_1(s_2-s_3-1)\zeta_{MT,2}(s_1+1,s_2-1,s_3)
    \\
    &\qquad
    -s_1(s_2-1)\zeta_{MT,2}(s_1+1,s_2,s_3-1)
    +s_1s_3\zeta_{MT,2}(s_1+1,s_2-2,s_3+1)
    \\
    &\qquad+(s_2-1)s_3\zeta_{MT,2}(s_1-1,s_2,s_3+1)
    +s_3(s_3+1)\zeta_{MT,2}(s_1-1,s_2-1,s_3+2)
    =0,
  \end{aligned}
  \\
  \label{eq:triv2}
  \begin{aligned}
    &(s_3(s_3+1)-2 s_1 s_3-2 s_2 s_3+2 s_1 s_2)\zeta_{MT,2}(s_1,s_2,s_3)%
    \\
    &\qquad
    +s_3 (s_1+2 s_2-2 s_3-2)\zeta_{MT,2}(s_1-1,s_2,s_3+1)
    \\
    &\qquad
    +s_3 (2 s_1+s_2-2 s_3-2)\zeta_{MT,2}(s_1,s_2-1,s_3+1)
    \\
    &\qquad
    +s_1 s_3\zeta_{MT,2}(s_1+1,s_2-2,s_3+1)
    +s_2 s_3\zeta_{MT,2}(s_1-2,s_2+1,s_3+1)
    \\
    &\qquad
    +s_1 (s_2-2 s_3)\zeta_{MT,2}(s_1+1,s_2-1,s_3)
    +s_2 (s_1-2 s_3)\zeta_{MT,2}(s_1-1,s_2+1,s_3)
    \\
    &\qquad 
    -s_1 (2 s_2-s_3)\zeta_{MT,2}(s_1+1,s_2,s_3-1)
    -s_2 (2 s_1-s_3)\zeta_{MT,2}(s_1,s_2+1,s_3-1)
    \\
    &\qquad
    +s_3(s_3+1)\zeta_{MT,2}(s_1-2,s_2,s_3+2)
    +s_3(s_3+1)\zeta_{MT,2}(s_1,s_2-2,s_3+2)
    \\
    &\qquad
    +s_1 s_2\zeta_{MT,2}(s_1+1,s_2+1,s_3-2)
    +2 s_3(s_3+1)\zeta_{MT,2}(s_1-1,s_2-1,s_3+2)
    =0.
  \end{aligned}
\end{gather}
The coefficients of $a$ and of $b$ give the same identity \eqref{eq:triv1}
(because of the symmetry of $s_1$ and $s_2$ in \eqref{Def-MT}), while
\eqref{eq:triv2} follows from the coefficient of $ab$.

However it should be noted that
each coefficient of $s_j$ in
\eqref{eq:triv1} and \eqref{eq:triv2} 
can be shown to be equal to 0 by partial fractional decompositions.
Hence these equations do not yield new relations.
Similarly in general cases,
it may be expected that only the constant term will give a non-trivial result.
\end{exa}

The following example can be regarded as a root-theoretic generalization of 
Example \ref{exa:MT2}, because $\zeta_{MT,2}(s_1,s_2,s_3)$ is the 
zeta-function of the root system of type $A_2$.

\begin{exa}
In the case of zeta-functions of root systems (cf. \cite{KMT10}), we have
 \begin{gather}
    (\xi_k)=(\xi_k)_{1\leqslant  k\leqslant  r},\qquad
    (\beta_\alpha)=(\langle\alpha^\vee,\rho\rangle)_{\alpha\in\Delta_+},
    \\
    (c_{m\alpha})=
    \begin{pmatrix}
      I_r & 0
    \end{pmatrix},
    \qquad
    (\gamma_{\alpha
k})=(\langle\alpha^\vee,\lambda_k\rangle)_{\alpha\in\Delta_+,1\leqslant 
k\leqslant  r},
  \end{gather}
where
$I_r$ is the $r\times r$ identity matrix,
$\Delta_+=\{\alpha_1,\ldots,\alpha_r,\ldots\}$ is the set of all positive
roots in a given root system, whose
first $r$ elements $\alpha_1,\ldots,\alpha_r$
are fundamental roots, $d=|\Delta_+|$, $\rho$ is the Weyl
vector, and $\lambda_1,\cdots,\lambda_r$ are fundamental weights.  
Thus
\begin{equation}
  \begin{split}
    G(\bs{u},\bs{v})&=
    \prod_{k=1}^r\Bigl(1-\delta(k)\Bigl( 
    1+\sum_{\alpha\in\Delta_+} c_{k\alpha} (v_\alpha^{-1}-\beta_\alpha)
    \Bigr)
    \Bigl(\sum_{\alpha\in\Delta_+} \gamma_{\alpha k}u_\alpha v_\alpha\Bigr)\Bigr)
    \\
    &=
    \prod_{k=1}^r\Bigl(1-\delta(k)\Bigl( 
    1+v_{\alpha_k}^{-1}-\langle\alpha_k^\vee,\rho\rangle
    \Bigr)
    \Bigl(\sum_{\alpha\in\Delta_+} \langle\alpha^\vee,\lambda_k\rangle u_\alpha v_\alpha\Bigr)\Bigr)
    \\
    &=
    \prod_{k=1}^r\Bigl(1-\delta(k)\sum_{\alpha\in\Delta_+} \langle\alpha^\vee,\lambda_k\rangle u_\alpha v_\alpha v_{\alpha_k}^{-1}
    \Bigr).
  \end{split}
\end{equation}
In particular, if
 $\xi_k=1$ ($1\leqslant  k\leqslant  r$), then
\begin{equation}
    G(\bs{u},\bs{v})
    =
    \prod_{k=1}^r\Bigl(1-\sum_{\alpha\in\Delta_+} \langle\alpha^\vee,\lambda_k\rangle u_\alpha v_\alpha v_{\alpha_k}^{-1}\Bigr).
\end{equation}

\end{exa}

\section{Special values of $\zeta_2^{\rm des}$ at any integer points}\label{sec-5}

The multiple zeta-function of Euler-Zagier type defined by \eqref{EZ-zeta}
can be meromorphically
continued to the whole complex space with many singularities (see \cite{AET01}).
In the case $r=2$, the singularities of $\zeta_2(s_1,s_2)$ are located on 
$$s_2=1,\ s_1+s_2=2,1,0,-2,-4,-6,\ldots$$
(\cite[Theorem 1]{AET01}), which implies that its special values of many integer points cannot be determined. 

Here we consider the desingularized double zeta-function of Euler-Zagier type defined by
$$\zeta^{\rm des}_2(s_1,s_2)=\zeta^{\rm des}_2(s_1,s_2;1,1;1,0,1,1;1,1)$$
in \eqref{def-MZF-2} with $(r,d)=(2,2)$. 
We showed in \cite[(4.3)]{FKMT1} that 
\begin{align}
  \zeta^{\rm des}_2(s_1,s_2)& =(s_1-1)(s_2-1)\zeta_2(s_1,s_2)\label{ex}\\
  & \quad +s_2(s_2+1-s_1)\zeta_2(s_1-1,s_2+1) -s_2(s_2+1)\zeta_2(s_1-2,s_2+2), \notag
\end{align}
which is entire. 
Therefore its special values of all integer points can be determined, though each
term on the right-hand side has singularities.
We give their explicit expressions as follows. Note that a part of the examples mentioned below were already introduced in \cite[Examples 4.7 and 4.9]{FKMT1} with no proof. 

First we consider the case $s_2\in \mathbb{Z}_{\leqslant  0}$. We prepare the following lemma.

\begin{lmm}\label{L-5-1} For $N\in \mathbb{N}_{0}$,
\begin{align}
\zeta_2(s,-N)&=-\frac{1}{N+1}\zeta(s-N-1)+\sum_{k=0}^{N}\binom{N}{k}\zeta(s-N+k)\zeta(-k),  \label{E-5-1}\\
\zeta_2(-N,s) &=\frac{1}{N+1}\zeta(s-N-1)-\sum_{k=0}^{N}\binom{N}{k}\zeta(s-N+k)\zeta(-k) \label{E-5-2}\\
&\quad +\zeta(s)\zeta(-N)-\zeta(s-N) \notag
\end{align}
hold for $s\in \mathbb{C}$ except for singularities.
\end{lmm}

\begin{proof}
It follows from \cite[(4.4)]{Mat02} that 
\begin{align}
\zeta_2(s_1,s_2)&=\frac{1}{s_2-1}\zeta(s_1+s_2-1)+\sum_{k=0}^{M-1}\binom{-s_2}{k}\zeta(s_1+s_2+k)\zeta(-k)\label{ill}\\
& \ +\frac{1}{2\pi i\Gamma(s_2)}\int_{(M-\varepsilon)}\Gamma(s_2+z)\Gamma(-z)\zeta(s_1+s_2+z)\zeta(-z)dz \notag
\end{align}
for $M\in \mathbb{N}$ and $(s_1,s_2)\in \mathbb{C}^2$ with $\Re s_2>-M+\varepsilon$, $\Re (s_1+s_2)>1-M+\varepsilon$ for any small $\varepsilon>0$. 
Setting $(s_1,s_2)=(s,-N)$ and $M= N+1$ in \eqref{ill}, we see that 
\eqref{E-5-1} holds for any $s\in \mathbb{C}$ except for singularities because the both sides of \eqref{E-5-1} can be continued meromorphically to $\mathbb{C}$.
Next, using the well-known relation
$$\zeta_2(s_1,s_2)+\zeta_2(s_2,s_1)=\zeta(s_1)\zeta(s_2)-\zeta(s_1+s_2),$$
we can immediately obtain \eqref{E-5-2}.
\end{proof}

\begin{exa}
From \eqref{E-5-1} and \eqref{E-5-2}, we have
\begin{align}
& \zeta_2(s,0)=-\zeta(s-1)-\frac{1}{2}\zeta(s), \label{form-a}\\
& \zeta_2(s,-1)=-\frac{1}{2}\zeta(s-2)-\frac{1}{2}\zeta(s-1)-\frac{1}{12}\zeta(s),\label{form-b}\\
& \zeta_2(0,s)=\zeta(s-1)-\zeta(s), \label{form-1}\\
& \zeta_2(-1,s)=\frac{1}{2}\left\{\zeta(s-2)-\zeta(s-1)\right\}.\label{form-2}
\end{align}
\end{exa}

\begin{prp} \label{P-5-2} \ For $s\in \mathbb{C}$ and $N\in \mathbb{N}_{0}$, 
\begin{align}
&\zeta_2^{\rm des}(s,-N) =-\sum_{k=0}^{N}\binom{N}{k}(k+1)(s-N+k-1)\zeta(s-N+k)\zeta(-k). \label{E-5-3}
\end{align}
\end{prp}

\begin{proof}
From \eqref{ex} we have
\begin{align*}
\zeta_2^{\rm des}(s,-N)&=(s-1)(-N-1)\zeta_2(s,-N)-N(-N+1-s)\zeta_2(s-1,-N+1)\\
& \ +N(-N+1)\zeta_2(s-2,-N+2).
\end{align*}
Substituting \eqref{E-5-1} with $(s,-N)$, $(s-1,-N+1)$ and $(s-2,-N+2)$ 
into the right-hand side of the above equation, we have
\begin{align*}
\zeta_2^{\rm des}(s,-N)&=\sum_{k=0}^N \bigg\{ (s-1)(-N-1)\binom{N}{k}-N(-N+1-s)\binom{N-1}{k}\\
& \qquad +N(-N+1)\binom{N-2}{k}\bigg\}\zeta(s-N+k)\zeta(-k),
\end{align*}
and the right-hand side of the above formula can be transformed to the right-hand side of \eqref{E-5-3}.
\end{proof}

\ 

\begin{exa}\label{Ex-5-3}
Setting $N=3$ in \eqref{E-5-3}, we obtain
\begin{equation}
\zeta_2^{\rm des}(s,-3)=\frac{s-4}{2}\zeta(s-3)+\frac{s-3}{2}\zeta(s-2)-\frac{s-1}{30}\zeta(s). \label{E-5-4}
\end{equation}
For example, 
\begin{align*}
& \zeta_2^{\rm des}(1,-3)=\frac{1}{20},\qquad\qquad\qquad
\zeta_2^{\rm des}(2,-3)=\frac{1}{3}-\frac{1}{30}\zeta(2),\\
& \zeta_2^{\rm des}(3,-3)=\frac{3}{4}-\frac{1}{15}\zeta(3),\qquad\;
 \zeta_2^{\rm des}(4,-3)=\frac{1}{2}+\frac{1}{2}\zeta(2)-\frac{1}{10}\zeta(4).
\end{align*}
Also we have
\begin{align*}
 \zeta_2^{\rm des}(0,0)=\frac{1}{4},\quad
 \zeta_2^{\rm des}(-1,-1)=\frac{1}{36},\quad
 \zeta_2^{\rm des}(0,-2)=\frac{1}{18}.
\end{align*}
\end{exa}

\begin{prp} \label{P-5-4} \ For $s\in \mathbb{C}$ and $N\in \mathbb{N}_{0}$, 
\begin{align}
& \zeta_2^{\rm des}(-N,s) \label{E-5-5}\\
& =\frac{(s-N-3)(s-N-2)}{(N+3)(N+2)}\zeta(s-N-1)\notag\\
& \ +\sum_{k=0}^{N+1}\frac{(ks+N-k+2)(s-N+k-1)}{N+2}\binom{N+2}{k}\zeta(s-N+k)\zeta(-k)\notag\\
& \ -(N+1)(s-1)\zeta(s)\zeta(-N)+s(s+1+N)\zeta(s+1)\zeta(-N-1)\notag\\
& \ +(s-N-1)\zeta(s-N).\notag
\end{align}
\end{prp}

\begin{proof}
From \eqref{ex}, we have
\begin{align*}
\zeta_2^{\rm des}(-N,s)&=(-N-1)(s-1)\zeta_2(-N,s)+s(s+1+N)\zeta_2(-N-1,s+1)\\
& \ -s(s+1)\zeta_2(-N-2,s+2).
\end{align*}
Similar to the proof of Proposition \ref{P-5-2}, 
substituting \eqref{E-5-2} with $(-N,s)$, $(-N-1,s+1)$ and $(-N-2,s+2)$ 
into the right-hand side of the above equation, we can obtain \eqref{E-5-5}.
Note that, in this case, we apply \eqref{E-5-2} with the sum on the right-hand side 
from 0 to $N+2$, but the term corresponding to $k=N+2$ is canceled and does not
appear in the final statement.
\end{proof}

\begin{exa}\label{Ex-5-5}
Setting $N=1$ in \eqref{E-5-5}, we have
\begin{equation}
\zeta_2^{\rm des}(-1,s)=\frac{(s-4)(s-3)}{12}\zeta(s-2)+\frac{s-2}{2}\zeta(s-1)-\frac{s(s-1)}{12}\zeta(s).\label{E-5-6}
\end{equation}
For example, 
\begin{align*}
\zeta_2^{\rm des}(-1,1)& 
=\frac{1}{8},\qquad
\zeta_2^{\rm des}(-1,2) 
=\frac{5}{12}-\frac{1}{6}\zeta(2),\\
\zeta_2^{\rm des}(-1,3)& =-\frac{1}{12}+\frac{1}{2}\zeta(2)-\frac{1}{2}\zeta(3).
\end{align*}
\end{exa}

Next we consider $\zeta_2^{\rm des}(N,1)$ $(N\in \mathbb{N})$. 
From \eqref{ex} with $s_1=N\in \mathbb{Z}_{>1}$ and $s_2\to 1$, we have
$$\zeta_2^{\rm des}(N,1)=(N-1)\lim_{s_2\to 1}(s_2-1)\zeta_2(N,s_2)+(2-N)\zeta_2(N-1,2)-2\zeta_2(N-2,3).$$
We know from Arakawa and Kaneko \cite[Proposition 4]{AK99} that
\begin{align}
\zeta_2(N,s)& =\frac{\zeta(N)}{s-1}+O(1) \quad (N\in \mathbb{Z}_{>1}).\label{E-5-7}
\end{align}
Thus we obtain the following.

\begin{prp}\label{P-5-6}
For $N\in \mathbb{N}_{>1}$,
\begin{align}
\zeta_2^{\rm des}(N,1)&=(N-1)\zeta(N)+(2-N)\zeta_2(N-1,2)-2\zeta_2(N-2,3). \label{E-5-9}
\end{align}
\end{prp}

\begin{exa}\label{Ex-5-7}
Using well-known results for double zeta-values, we obtain
\begin{align*}
\zeta_2^{\rm des}(2,1)&=\zeta(2)-2\zeta_2(0,3)=2\zeta(3)-\zeta(2),\\
\zeta_2^{\rm des}(3,1)&=2\zeta(3)-\zeta_2(2,2)-2\zeta_2(1,3)=2\zeta(3)-\frac{5}{4}\zeta(4),\\
\zeta_2^{\rm des}(4,1)&=3\zeta(4)-2\zeta_2(3,2)-2\zeta_2(2,3)=3\zeta(4)+2\zeta(5)-2\zeta(2)\zeta(3),
\end{align*}
where we note $\zeta_2(0,3)=\zeta(2)-\zeta(3)$. 
\end{exa}

The case $N=1$ should be treated separately.

\begin{prp}\label{P-5-9}
$$\zeta^{\rm des}_2(1,1)=\frac{1}{2}.$$
\end{prp}

\begin{proof}
Denote the first, the second and the third term on the right-hand side of \eqref{ex} by $I_1,I_2$ and $I_3$, respectively. Setting $M=1$ in \eqref{ill}, we have
\begin{equation}
\lim_{s_2\to 1}\lim_{s_1\to 1}I_1=\lim_{s_2\to 1}(s_2-1)\lim_{s_1\to 1}(s_1-1)\zeta_2(s_1,s_2)=0. \label{I1}
\end{equation}
Using \eqref{form-1} and \eqref{form-2}, we obtain
\begin{align}
\lim_{s_2\to 1}\lim_{s_1\to 1}(I_2+I_3)&=\lim_{s_2 \to 1}\left\{s_2^2\zeta_2(0,s_2+1) -s_2(s_2+1)\zeta_2(-1,s_2+2)\right\}\label{I2}\\
& =\lim_{s_2 \to 1}\left(s_2^2-\frac{s_2(s_2+1)}{2}\right)\{\zeta(s_2)-\zeta(s_2+1)\}\notag\\
& =\lim_{s_2 \to 1}\frac{s_2}{2}(s_2-1)\{\zeta(s_2)-\zeta(s_2+1)\}=\frac{1}{2}. \notag
\end{align}
From \eqref{I1} and \eqref{I2}, we obtain the assertion.
Note that, since $\zeta_2^{\rm des}(s_1,s_2)$ is entire, the final result does
not depend on the choice how to take the limit.
\end{proof}

%

\section{$p$-adic multiple star polylogarithm for indices with arbitrary integers}\label{pMLF=pMPL}

Now we proceed to our second main topic of the present paper.
Our aim is to extend the result  of
\cite[Theorem 3.41]{FKMT2} to the case of indices with 
arbitrary (not necessarily all positive) integers 
(Theorem \ref{L-Li theorem-2}), which is a $p$-adic analogue of the equation \eqref{MPL-03}.

First we prepare ordinary notation. For a prime number $p$, let $\mathbb{Z}_p$, $\mathbb{Q}_p$, $\overline{\mathbb{Q}}_p$ and $\mathbb{C}_p$ 
be the set of $p$-adic integers, $p$-adic numbers, the algebraic closure of $\mathbb{Q}_p$ and the $p$-adic completion of $\overline{\mathbb{Q}}_p$ 
respectively.
For $a$ in $\textbf{P}^{1}(\mathbb{C}_p)\left(={\mathbb{C}}_p\cup \{\overline{ \infty}\}\right)$, $\bar a$ means the image ${\rm red}(a)$ by the reduction map 
${\rm red}:\textbf{P}^{1}(\mathbb{C}_p) \to \textbf{P}^{1}({\overline{\mathbb{F}}_p})\left(=\overline{\mathbb{F}}_p\cup \{\overline{ \infty}\}\right),$
where $\overline{\mathbb{F}}_p$ is the algebraic closure of ${\mathbb{F}}_p$. 
Denote by $|\cdot |_p$ the $p$-adic absolute value, and by $\mu_c$ 
the group of $c$\,th roots of unity in $\mathbb{C}_p$ for $c\in\mathbb{N}$. 
We put $q=p$ if $p\neq 2$ and $q=4$ if $p=2$.
We denote by $\omega:\mathbb Z_p^\times\to \mathbb Z_p^\times$  the Teichm\"{u}ller character and define
$\langle x\rangle:=x/\omega(x)$ for $x\in\mathbb Z_p^\times$.

We recall that,
for $r\in \mathbb{N}$, $k_1,\ldots,k_r\in \mathbb{Z}$
and $c\in \mathbb{N}_{>1}$ with $(c,p)=1$,
the {\bf $p$-adic multiple $L$-function} of depth $r$, 
a $\mathbb{C}_p$-valued function
on  
$$(s_j)\in \mathfrak{X}_r\left(q^{-1}\right)
:=\left\{(s_1,\ldots,s_r)\in \mathbb{C}_p^r\,\big|\,|s_j|_p<qp^{-1/(p-1)}\ (1\leqslant  j\leqslant  r)\right\},$$
is defined in \cite{FKMT2} by
\begin{align*}
&{L_{p,r}(s_1,\ldots,s_r;\omega^{k_1},\ldots,\omega^{k_r};c)}\\
&\quad :=\int_{\left( \mathbb{Z}_p^r\right)'
}\langle x_1 \rangle^{-{s_1}}\langle x_1+ x_2\rangle^{-{s_2}}\cdots \langle \sum_{j=1}^{r}x_{j} \rangle^{-{s_r}}
\omega^{k_1}(x_1)\cdots\omega^{k_r}( \sum_{j=1}^{r}x_{j}) \prod_{j=1}^{r}d\mm(x_j),
\end{align*}
where
$\left( \mathbb{Z}_p^r\right)'
:=\bigg\{ (x_j)\in \mathbb{Z}_p^r \,\bigg|\,p\nmid x_1,\ p\nmid (x_1+x_2),\ldots,\ p\nmid \sum_{j=1}^{r}x_j\,\bigg\}$,
and $\mm$ is the $p$-adic measure given in \cite[\S 1]{FKMT2}.
The function is equal to 
$L_{p,r}(s_1,\ldots,s_r;\omega^k_1,\ldots,\omega^k_r;1,\ldots,1;c)$ in \cite[Definition 1.16]{FKMT2}. When $r=1$, we have
\begin{equation}
L_{p,1}(s;\omega^{k-1};c)=(\langle c \rangle^{1-s}\omega^{k}(c)-1)L_p(s;\omega^{k}), \label{KL-pMLF}
\end{equation}
where $L_p(s;\omega^{k})$ is the Kubota-Leopoldt $p$-adic $L$-function (see \cite[Example 1.19]{FKMT2}). 

The $p$-adic rigid TMSPL can be defined for  indices with arbitrary integers 
in the same way as \cite[Definition 3.4]{FKMT2}:
Let $n_1,\dots,n_r\in \mathbb{Z}$ and $\xi_1,\dots,\xi_{r}\in\mathbb{C}_p$ with $|\xi_j|_p\leqslant  1$ ($1\leqslant  j\leqslant  r$).
The {\bf $p$-adic rigid TMSPL}
\footnote{
TMSPL stands for the twisted multiple star polylogarithm.
Here 'star' means that we add equalities in the running indices of the summation.
}
is defined by the following 
$p$-adic power series:
\begin{equation}\label{series expression for ell}
\ell^{(p),\star}_{n_1,\dots,n_r}(\xi_1,\dots,\xi_{r},z):=
\underset{(k_1,p)=\cdots=(k_r,p)=1}{\underset{0<k_1\leqslant\cdots\leqslant k_{r}}{\sum}}
\frac{\xi_1^{k_1}\cdots\xi_r^{k_r}}{k_1^{n_1}\cdots k_r^{n_r}}z^{k_r}
\end{equation}
which converges for $z\in ]\bar{0}[=\{x\in\mathbb{C}_p\bigm| \ |x|_p<1\}$ 
by $|\xi_j|_p\leqslant  1$ for $1\leqslant  j\leqslant  r$.

When $|\xi_j|_p= 1$ for all $1\leqslant  j\leqslant  r$,
by the completely same  way as the arguments in \cite[\S 3]{FKMT2}, 
we can show that it can be extended to a rigid analytic function
(consult \cite[\S 3.1]{FKMT2})
on ${\bf P}^1({\mathbb C}_p) - ]S[$ 
with
\begin{equation}\label{S0}
S:=\{
\overline{\xi_{r}^{-1}},\overline{(\xi_{r-1}\xi_{r})^{-1}},
\dots,\overline{(\xi_1\cdots\xi_{r})^{-1}}\}.
\end{equation}
Namely, 
$$
\ell^{(p),\star}_{n_1,\dots,n_r}(\xi_1,\dots,\xi_{r};z)
\in A^{\rm{rig}}( {\bf P}^1\setminus S
).
$$
We also note that
\begin{equation}\label{l0}
\ell^{(p),\star}_{n_1,\dots,n_r}(\xi_1,\dots,\xi_{r};0)
=\ell^{(p),\star}_{n_1,\dots,n_r}(\xi_1,\dots,\xi_{r};\infty)=0,
\end{equation}
and the following equality:

\begin{prp}\label{L-ell prop}
For  $n_1,\dots,n_r\in \mathbb{Z}$
and $c\in \mathbb{N}_{>1}$ with $(c,p)=1$,
\begin{equation*}\label{L-ell-formula}
L_{p,r}(n_1,\dots,n_r;\omega^{-n_1},\dots,\omega^{-n_r};c)= \\
\underset{\xi_1\cdots\xi_r\neq 1, \ \dots, \ \xi_{r-1}\xi_r\neq 1, \ \xi_r\neq 1}{\sum_{\xi_1^c=\cdots=\xi_r^c=1}}
\ell^{(p),\star}_{n_1,\dots,n_r}(\xi_1,\dots,\xi_r;1).
\end{equation*}
\end{prp}

The $p$-adic partial TMSPL can also be defined for indices 
with arbitrary integers 
in the same way as \cite[Definition 3.4]{FKMT2}:
Let  $n_1,\dots,n_r\in \mathbb{Z}$ and $\xi_1,\dots,\xi_{r}\in\mathbb{C}_p$ with $|\xi_j|_p\leqslant  1$ ($1\leqslant  j\leqslant  r$).
Let  $\alpha_1,\dots,\alpha_r\in \mathbb{N}$ with $0<\alpha_j<p$
 ($1\leqslant  j\leqslant  r$).
The {\bf $p$-adic partial TMSPL} 
$\ell^{\equiv (\alpha_1,\dots,\alpha_r),(p),\star}_{n_1,\dots,n_r}(\xi_1,\dots\xi_{r};z)$ is defined by the following 
$p$-adic power series:
\begin{equation}\label{series expression for partial double}
\ell^{\equiv (\alpha_1,\dots,\alpha_r),(p),\star}_{n_1,\dots,n_r}(\xi_1,\dots,\xi_{r};z):=
\underset{k_1\equiv \alpha_1,\dots ,k_r\equiv\alpha_r \bmod p
}{\underset{0<k_1\leqslant\cdots \leqslant k_r}{\sum}}
\frac{\xi_1^{k_1}\cdots\xi_{r}^{k_{r}}}{k_1^{n_1}\cdots k_r^{n_r}}z^{k_r}
\end{equation}
which converges for $z\in ]\bar{0}[$.

When $|\xi_j|_p= 1$ for all $1\leqslant  j\leqslant  r$,
by the completely same way as the arguments in \cite[\S 3.2]{FKMT2}, 
we can show that it is a rigid analytic function on
${\bf P}^1({\mathbb C}_p) - ]S[$.
Namely, 
\begin{equation}\label{partial-rig}
\ell^{\equiv (\alpha_1,\dots,\alpha_r),(p),\star}_{n_1,\dots,n_r}(\xi_1,\dots,\xi_{r};z)\in 
A^{\mathrm{rig}}( {\bf P}^1\setminus S).
\end{equation}
We have 
\begin{equation}\label{partiall0}
\ell^{\equiv (\alpha_1,\dots,\alpha_r),(p),\star}_{n_1,\dots,n_r}(\xi_1,\dots,\xi_{r};0)
=\ell^{\equiv (\alpha_1,\dots,\alpha_r),(p),\star}_{n_1,\dots,n_r}(\xi_1,\dots,\xi_{r};\infty)=0
\end{equation}
by the equality
\begin{equation*}\label{partial ell-ell}
\ell^{\equiv (\alpha_1,\dots,\alpha_r),(p),\star}_{n_1,\dots,n_r}(\xi_1,\dots,\xi_{r};z)
=\frac{1}{p^r}\sum_{\rho_1^p=\cdots=\rho_r^p=1} \rho_1^{-\alpha_1}\cdots\rho_r^{-\alpha_r}
\ell^{(p),\star}_{n_1,\dots,n_r}(\rho_1\xi_1,\dots,\rho_{r}\xi_{r}; z).
\end{equation*}
We also note
 \begin{equation}\label{ell-partial ell}
\ell^{(p),\star}_{n_1,\dots,n_r}(\xi_1,\dots,\xi_{r};z)=
\sum_{0<\alpha_1,\dots, \alpha_r<p}
\ell^{\equiv (\alpha_1,\dots,\alpha_r),(p),\star}_{n_1,\dots,n_r}(\xi_1,\dots,\xi_{r};z).
\end{equation}

The following formulas are extensions of \cite[Lemma 3.19]{FKMT2}
to the case of indices 
with arbitrary integers. 

\begin{lmm}\label{differential equations}
Let  $n_1,\dots,n_r\in \mathbb{Z}$, 
$\xi_1,\dots,\xi_{r}\in\mathbb{C}_p$ with $|\xi_j|_p\leqslant  1$ 
$(1\leqslant  j\leqslant  r)$
and $\alpha_1,\dots,\alpha_r\in \mathbb{N}$ with $0<\alpha_j<p$
$(1\leqslant  j\leqslant  r)$.

{\rm (i)}\ For any index $(n_1,\dots,n_r)$,  
$$\frac{d}{dz}\ell^{\equiv (\alpha_1,\dots,\alpha_r),(p),\star}_{n_1,\dots,n_r}(\xi_1,\dots,\xi_{r};z)
=\frac{1}{z}\ell^{\equiv (\alpha_1,\dots,\alpha_r),(p),\star}_{n_1,\dots,n_r-1}(\xi_1,\dots,\xi_{r};z).
$$

{\rm (ii)}\ For $n_r=1$ and $r\neq 1$,
\begin{align*}
\lefteqn{\frac{d}{dz}\ell^{\equiv (\alpha_1,\dots,\alpha_r),(p),\star}_{n_1,\dots,n_r}(\xi_1,\dots,\xi_{r};z)}\\
&=
\begin{cases}
&\frac{\xi_r(\xi_r z)^{\alpha_r-\alpha_{r-1}-1}}{1-(\xi_r z)^p}
\ell^{\equiv (\alpha_1,\dots,\alpha_{r-1}),(p),\star}_{n_1,\dots,n_{r-1}}(\xi_1,\dots,\xi_{r-2},\xi_{r-1};\xi_r z)  \\
&\qquad\qquad\qquad\qquad\qquad\qquad \text{if}\quad \alpha_r\geqslant\alpha_{r-1} , \\
&\frac{\xi_r(\xi_r z)^{\alpha_r-\alpha_{r-1}+p-1}}{1-(\xi_r z)^p}\ell^{\equiv (\alpha_1,\dots,\alpha_{r-1}),(p),\star}_{n_1,\dots,n_{r-1}}(\xi_1,\dots,\xi_{r-2},\xi_{r-1};\xi_r z)  \\
&\qquad\qquad\qquad\qquad\qquad\qquad \text{if}\quad \alpha_r < \alpha_{r-1}. \\
\end{cases}
\end{align*}

{\rm (iii)}\ For $n_r=1$ and $r=1$ with $\xi_1=\xi$ and  $\alpha_1=\alpha$,
 $$\frac{d}{dz}\ell^{\equiv \alpha, (p),\star}_{1}(\xi;z)=
\frac{\xi (\xi z)^{\alpha-1}}{1-(\xi z)^p}.
$$

\end{lmm}

\begin{proof}
They can be proved by direct computations.
\end{proof}

The following result is an extension of \cite[Theorem 3.21]{FKMT2} to the case of indices 
with arbitrary integers.

\begin{prp}\label{rigidness II}
Let  $n_1,\dots,n_r\in \mathbb{Z}$, 
$\xi_1,\dots,\xi_{r}\in\mathbb{C}_p$ with $|\xi_j|_p= 1$ $(1\leqslant  j\leqslant  r)$
and $\alpha_1,\dots,\alpha_r\in \mathbb{N}$ with $0<\alpha_j<p$
$(1\leqslant  j\leqslant  r)$.
Set $S$ as in \eqref{S0}.
The function 
$\ell^{\equiv (\alpha_1,\dots,\alpha_r),(p),\star}_{n_1,\dots,n_r}(\xi_1,\dots,\xi_{r}; z)$
is an overconvergent function on 
${\bf P}^1\setminus S$.
Namely, 
$$\ell^{\equiv (\alpha_1,\dots,\alpha_r),(p),\star}_{n_1,\dots,n_r}(\xi_1,\dots,\xi_{r};z)\in
 A^\dag( {\bf P}^1\setminus S
).$$
\end{prp}

Here $ A^\dag( {\bf P}^1\setminus S)$ means the space of overconvergent functions on ${\bf P}^1\setminus S$ (consult \cite[Notation 3.13]{FKMT2}).

\begin{proof}
The proof of \cite[Theorem 3.21]{FKMT2}  was done by the induction on the weight
but here it is achieved by the induction on the depth $r$.

(i) Assume that $r=1$.
By \cite[Theorem 3.21]{FKMT2}, we  know 
$\ell^{\equiv \alpha_1,(p),\star}_{n_1}(\xi_1;z)\in A^\dag( {\bf P}^1\setminus S)$
when $n_1>0$.
When $n_1\leqslant  0$, 
by Lemma \ref{differential equations} (i) and (iii) we know that 
the function is a rational function and
the degree of whose numerator 
is less than that of whose denominator
which is a power of $1-(\xi_1 z)^p$,
which implies that the poles of the function
are of the form $\zeta_p/\xi_1$ ($\zeta_p\in\mu_p$).

(ii) Assume that $r>1$ and $n_r=1$.
We put
\begin{equation*}
S_\infty=S\cup \{\overline{ \infty}\}
\qquad \text{ and } \qquad
S_{\infty,0}=S\cup \{\overline{ \infty}\}\cup\{\bar 0\}
\end{equation*}
and take a lift $\{\widehat{s_0}, \widehat{s_1}, \dots, \widehat{s_d}\}$  
of $S_{\infty,0}$  with
$\widehat{s_0}=\infty$  and 
$\widehat{s_1}=0$.
Put
$$
\beta(z):=
\begin{cases}
\frac{\xi_r (\xi_r z)^{\alpha_r-\alpha_{r-1}-1}}{1-(\xi_r z)^p} &\text{if}\quad \alpha_r\geqslant\alpha_{r-1} ,\\
\frac{\xi_r (\xi_r z)^{\alpha_r-\alpha_{r-1}+p-1}}{1-(\xi_r z)^p}
&\text{if}\quad \alpha_r< \alpha_{r-1} .\\
\end{cases}
$$
By our assumption
$$
\ell^{\equiv (\alpha_1,\dots,\alpha_{r-1}),(p),\star}_{n_1,\dots,n_{r-1}}(\xi_1,\dots,\xi_{r-2},\xi_{r-1};\xi_r z) 
\in A^\dag( {\bf P}^1 
\setminus\{
\overline{\xi_{r}^{-1}},\dots,\overline{(\xi_1\cdots\xi_{r})^{-1}}\} )
$$
and by the fact
$\beta(z)dz\in
\Omega^{\dag,1}({\bf P}^1
\setminus\{\overline{0},\overline{\infty},\overline{\xi_r^{-1}}\})$,
we have
\begin{equation*}
\ell^{\equiv (\alpha_1,\dots,\alpha_{r-1}),(p),\star}_{n_1,\dots,n_{r-1}}(\xi_1,\dots,\xi_{r-2},\xi_{r-1};\xi_r z) \cdot \beta(z)dz
\in\Omega^{\dag,1}( {\bf P}^1 \setminus S_{\infty,0}).
\end{equation*}
Put
\begin{equation}\label{put}
f(z):=\ell^{\equiv (\alpha_1,\dots,\alpha_{r-1}),(p),\star}_{n_1,\dots,n_{r-1}}(\xi_1,\dots,\xi_{r-2},\xi_{r-1};\xi_r z) \cdot \beta(z)
\in A^{\dag}( {\bf P}^1 \setminus S_{\infty,0}).
\end{equation}
For the symbols $A^\dag$ and $\Omega^{\dag,1}$, consult 
\cite[\S 3.2]{FKMT2}.

Since $\ell^{\equiv (\alpha_1,\dots,\alpha_r),(p),\star}_{n_1,\dots,n_r}(\xi_1,\dots,\xi_{r};z)$
belongs to
$
A^{\text{rig}}( {\bf P}^1\setminus S
) 
\Bigl(\subset
A^{\text{rig}}( {\bf P}^1\setminus S_{\infty,0}
 ) 
\Bigr)
$
by \eqref{partial-rig}
and it satisfies the differential equation in Lemma \ref{differential equations} (ii),
i.e. its differential  is equal to $f(z)$,
we have particularly,
in the expression of \cite[Lemma 3.14]{FKMT2},
\begin{equation}\label{expansion at 0}
a_m(\widehat{ s_1};f)=0 \qquad (m>0)
\end{equation}
(recall $\widehat{s_1}=0$) and
\begin{equation}\label{other residues}
a_1(\widehat{ s_l};f)=0 \qquad (2\leqslant  l\leqslant  d).
\end{equation}

By \eqref{put} and \eqref{expansion at 0},
$$
f(z)\in A^{\dag}( {\bf P}^1\setminus S_\infty
).
$$

By \eqref{other residues} and  \cite[Lemma 3.15]{FKMT2},
there exists a unique function
$F(z)$ in 
$
A^{\dag}( {\bf P}^1\setminus S_\infty
)
$,
i.e. a function $F(z)$ which is rigid analytic on an affinoid  $V$ containing
$$
{\bf P}^1({\mathbb C}_p) - \,]S_\infty[ \  = 
{\bf P}^1({\mathbb C}_p) - \,]\overline{ \infty},S[
$$
such that 
\begin{equation}\label{differential property}
F(0)=0 \quad \text{and} \quad
dF(z)=f(z)dz.
\end{equation}

Since $\ell^{\equiv (\alpha_1,\dots,\alpha_r),(p),\star}_{n_1,\dots,n_r}(\xi_1,\dots,\xi_{r}; z)$
is also a unique function in
$
A^{\text{rig}}( {\bf P}^1\setminus S
) 
$
satisfying \eqref{differential property},
the restrictions of both $F(z)$ and
$\ell^{\equiv (\alpha_1,\dots,\alpha_r),(p),\star}_{n_1,\dots,n_r}(\xi_1,\dots,\xi_{r}; z)$
to the subspace
${\bf P}^1({\mathbb C}_p) -]S_\infty[
$ must coincide, i.e.
$$
F(z)\Bigm|_{{\bf P}^1({\mathbb C}_p) - ]S_\infty[
}
\equiv
\ell^{\equiv (\alpha_1,\dots,\alpha_r),(p),\star}_{n_1,\dots,n_r}(\xi_1,\dots,\xi_{r};z)\Bigm|_{{\bf P}^1({\mathbb C}_p) - ]S_\infty[
}.
$$

Hence by the coincidence principle of rigid analytic functions (\cite[Proposition 3.3]{FKMT2}),
there is a rigid analytic function $G(z)$ on the union of $V$ and
${\bf P}^1({\mathbb C}_p) - ]S[$
whose restriction to $V$ is equal to $F(z)$
and whose restriction to 
${\bf P}^1({\mathbb C}_p) -]S[
$
is equal to
$\ell^{\equiv (\alpha_1,\dots,\alpha_r),(p),\star}_{n_1,\dots,n_r}(\xi_1,\dots,\xi_{r}; z)$.
So we can say that  
$$\ell^{\equiv (\alpha_1,\dots,\alpha_r),(p),\star}_{n_1,\dots,n_r}(\xi_1,\dots,\xi_{r};z)\in
A^{\text{rig}}( {\bf P}^1\setminus S
) 
$$
can be rigid analytically
extended to a bigger rigid analytic space
by $G(z)$.
Namely, 
$$\ell^{\equiv (\alpha_1,\dots,\alpha_r),(p),\star}_{n_1,\dots,n_r}(\xi_1,\dots,\xi_{r};z)\in
 A^\dag( {\bf P}^1\setminus S
).$$

(iii)  Assume that $r>1$ and $n_r<1$.
In our (ii) above,  we showed that 
\begin{equation}\label{nr1}
\ell^{\equiv (\alpha_1,\dots,\alpha_r),(p),\star}_{n_1,\dots,n_{r-1},1}(\xi_1,\dots,\xi_{r};z)\in
 A^\dag( {\bf P}^1\setminus S).
\end{equation}
Now showing that $\ell^{\equiv (\alpha_1,\dots,\alpha_r),(p),\star}_{n_1,\dots,n_r}(\xi_1,\dots,\xi_{r};z)\in
 A^\dag( {\bf P}^1\setminus S)$
is immediate,
which follows from
the differential equation in Lemma \ref{differential equations} (i)
and \eqref{partiall0}.

(iv)  Assume that $r>1$ and $n_r>1$.
The proof in this case can be achieved by the induction on $n_r$.
Recall that  we have \eqref{nr1} by our (ii) above.
By our assumption
$$
\ell^{\equiv (\alpha_1,\dots,\alpha_r),(p),\star}_{n_1,\dots,n_{r-1},n_r-1}(\xi_1,\dots,\xi_{r-1},\xi_r;z)\in A^\dag( {\bf P}^1\setminus S
)
$$
and by the fact $\frac{dz}{z}\in 
\Omega^{\dag,1}({\bf P}^1
\setminus\{\overline{\infty},\overline{0}\})$,
we have 
\begin{equation*}
\ell^{\equiv (\alpha_1,\dots,\alpha_r),(p),\star}_{n_1,\dots,n_{r-1},n_r-1}(\xi_1,\dots,\xi_{r-1},\xi_r;z)\frac{dz}{z}\in
\Omega^{\dag,1}( {\bf P}^1\setminus S_{\infty,0}
 ).
\end{equation*}
Put
\begin{equation*}
f(z):=\frac{1}{z}\ell^{\equiv (\alpha_1,\dots,\alpha_r),(p),\star}_{n_1,\dots,n_{r-1},n_r-1}(\xi_1,\dots,\xi_{r-1},\xi_r; z)\in
A^{\dag}( {\bf P}^1\setminus S_{\infty,0}
).
\end{equation*}
Then  it follows that
$$\ell^{\equiv (\alpha_1,\dots,\alpha_r),(p),\star}_{n_1,\dots,n_r}(\xi_1,\dots,\xi_{r}; z)\in
 A^\dag( {\bf P}^1\setminus S)$$
by the same arguments as those given in (ii) above.
\end{proof}

By \eqref{ell-partial ell} and Proposition \ref{rigidness II}, we have

\begin{crl}\label{rigidness 3}
Let  $n_1,\dots,n_r\in \mathbb{Z}$, 
$\xi_1,\dots,\xi_{r}\in\mathbb{C}_p$ with $|\xi_j|_p= 1$ $(1\leqslant  j\leqslant  r)$.
Set $S$ as in \eqref{S0}.
The function 
$\ell^{(p),\star}_{n_1,\dots,n_r}(\xi_1,\dots,\xi_{r}; z)$
is an overconvergent function on 
${\bf P}^1\setminus S$.
Namely, 
$\ell^{(p),\star}_{n_1,\dots,n_r}(\xi_1,\dots,\xi_{r};z)\in
A^\dag( {\bf P}^1\setminus S).$
\end{crl}

The $p$-adic TMSPL can also be defined for  indices 
with arbitrary integers 
in the same way as \cite[Definition 3.29]{FKMT2}:
Let  $n_1,\dots,n_r\in \mathbb{Z}$ and $\xi_1,\dots,\xi_{r}\in\mathbb{C}_p$ with $|\xi_j|_p\leqslant  1$ ($1\leqslant  j\leqslant  r$).
The {\bf $p$-adic TMSPL} 
$Li^{(p),\star}_{n_1,\dots,n_r}(\xi_1,\dots,\xi_r; z)$ is defined by the following 
$p$-adic power series:
\begin{equation}\label{series expression}
Li^{(p),\star}_{n_1,\dots,n_r}(\xi_1,\dots,\xi_r ;z):=
{\underset{0<k_1\leqslant\cdots\leqslant k_{r}}{\sum}}
\frac{\xi_1^{k_1}\cdots\xi_{r}^{k_{r}} z^{k_r}}{k_1^{n_1}\cdots k_r^{n_r}}
\end{equation}
which converges for $z\in ]\bar{0}[$ 
by $|\xi_j|_p\leqslant  1$ for $1\leqslant  j\leqslant  r$.
By direct computations one obtains the following differential equations which are
extensions of  \cite[Lemma 3.31]{FKMT2} to the case of indices 
with arbitrary integers.

\begin{lmm}\label{differential equations for Li}
Let $n_1,\dots,n_r\in \mathbb{Z}$,
$\xi_1,\dots,\xi_{r}\in\mathbb{C}_p$
with $|\xi_j|_p\leqslant  1$ $(1\leqslant  j\leqslant  r)$.

{\rm (i)}\ For any index $(n_1,\dots,n_r)$,  
$$\frac{d}{dz}Li^{(p),\star}_{n_1,\dots,n_r}(\xi_1,\dots,\xi_{r};z)
=\frac{1}{z}Li^{(p),\star}_{n_1,\dots,n_{r-1},n_r-1}(\xi_1,\dots,\xi_{r};z).$$

{\rm (ii)}\ For $n_r=1$ and $r\neq 1$,
$$\frac{d}{dz}Li^{(p),\star}_{n_1,\dots,n_r}(\xi_1,\dots,\xi_{r};z)=
\left\{\frac{\xi_r}{1-\xi_r z}+\frac{1}{z}\right\}Li^{(p),\star}_{n_1,\dots,n_{r-1}}(\xi_1,\dots,\xi_{r-2},\xi_{r-1};\xi_r z).$$

{\rm (iii)}\ For $n_r=1$ and $r=1$ with $\xi_1=\xi$,
$$\frac{d}{dz}Li^{(p),\star}_{1}(\xi;z)=
\frac{\xi}{1-\xi z}.$$
\end{lmm}

The following result is an extension of \cite[Theorem-Definition 3.32]{FKMT2} to the case of 
indices 
with arbitrary integers.

\begin{prp}\label{Coleman function theorem}
Fix a branch of the $p$-adic logarithm by $\varpi\in\mathbb C_p$.
Let $n_1,\dots,n_r\in \mathbb{Z}$,
$\xi_1,\dots,\xi_{r}\in\mathbb{C}_p$
with $|\xi_j|_p\leqslant  1$ $(1\leqslant  j\leqslant  r)$.
Put 
$$
S_r:=\{\bar 0,\overline{ \infty}, \overline{(\xi_r)^{-1}},\overline{(\xi_{r-1}\xi_r)^{-1}},\dots,\overline{(\xi_1\cdots\xi_r)^{-1}}\}
\subset  {\bf{P}}^{1}(\overline{\mathbb{F}}_p).
$$
Then the function
$Li^{(p),\star}_{n_1,\dots,n_r}(\xi_1,\dots,\xi_r; z)$
can be analytically continued as a 
Coleman function attached to  $\varpi\in\mathbb C_p$,
that is, 
$$
Li^{(p),\star,\varpi}_{n_1,\dots,n_r}(\xi_1,\dots,\xi_r; z)\in
A^\varpi_{\text{\rm Col}}
({\bf P}^1\setminus S_r
)
$$
whose restriction to $]\bar 0[$ is given by 
$Li^{(p),\star}_{n_1,\dots,n_r}(\xi_1,\dots,\xi_r; z)$ and
which is constructed
by the following iterated integrals:
\begin{equation}
Li^{(p),\star,\varpi}_{1}(\xi_1;z)=-\log^\varpi (1-\xi_1z)=\int_0^z\frac{\xi_1}{1-\xi_1t}dt, 
\label{intLi1}
\end{equation}
\begin{align}
&
Li^{(p),\star,\varpi}_{n_1,\dots,n_r}(\xi_1,\dots,\xi_r; z) \label{intLi2} 
\\
& \ \ =
\begin{cases}
\int_0^zLi^{(p),\star,\varpi}_{n_1,\dots,n_{r-1},n_r-1}(\xi_1,\dots,\xi_r; t)\frac{dt}{t} &\text{if}\quad n_r\neq 1 ,\\
\int_0^zLi^{(p),\star,\varpi}_{n_1,\dots,n_{r-1}}(\xi_1,\dots,\xi_{r-2},\xi_{r-1}\xi_r; t)
\{\frac{\xi_r}{1-\xi_r t}+\frac{1}{t}\}dt &\text{if}\quad n_r=1. 
\end{cases}
\notag
\end{align}
\end{prp}

Here $A^\varpi_{\textrm{Col}}({\bf P}^1\setminus S_r)$
means the space of Coleman functions of ${\bf P}^1\setminus S_r$
(consult \cite[Notation 3.25]{FKMT2}).

\begin{proof}
The proof of \cite[Theorem-Definition 3.32]{FKMT2}  was done by the induction on the weight
but here it is achieved by the induction on the depth $r$.

(i) Assume that $r=1$.
By \cite[Theorem-Definition 3.32]{FKMT2}, we  know 
$Li^{(p),\star,\varpi}_{n_1}(\xi_1; z)\in A^\varpi_{\textrm{Col}}({\bf P}^1\setminus S_1)$ when $n_1>0$.
When $n_1\leqslant  0$, it is immediate to see the assertion by 
the differential equation in Lemma \ref{differential equations for Li} (iii)
because differentials of Coleman functions are again 
Coleman functions.

(ii) Assume that $r>1$ and $n_r=1$.
Then by our induction assumption on $r$,
$Li^{(p),\star,\varpi}_{n_1,\dots,n_{r-1}}(\xi_1,\dots,\xi_{r-1}; z)\in A^\varpi_{\textrm{Col}}({\bf P}^1\setminus S_{r-1})$
and also $Li^{(p),\star,\varpi}_{n_1,\dots,n_{r-1}}(\xi_1,\dots,\xi_{r-1}; 0)=0$.
Hence $Li^{(p),\star,\varpi}_{n_1,\dots,n_{r-1}}(\xi_1,\dots,\xi_{r-1}; t)$
has a zero at $t=0$.
Therefore the integrand on the right-hand side of \eqref{intLi2} has no pole at $t=0$.
So the integration \eqref{intLi2} starting from $0$ makes sense and
whence we have
\begin{equation}\label{Li-Col-nr1}
Li^{(p),\star,\varpi}_{n_1,\dots,n_{r-1},1}(\xi_1,\dots,\xi_{r-1},\xi_r; z)\in A^\varpi_{\textrm{Col}}({\bf P}^1\setminus S_{r}).
\end{equation}

(iii) Assume that $r>1$ and $n_r<1$.
It is immediate to prove
$$Li^{(p),\star,\varpi}_{n_1,\dots,n_{r-1},n_r}(\xi_1,\dots,\xi_{r-1},\xi_r; z)\in A^\varpi_{\textrm{Col}}({\bf P}^1\setminus S_{r})$$
by \eqref{Li-Col-nr1} and
the differential equation in Lemma \ref{differential equations for Li} (i).

(iv) Assume that $r>1$ and $n_r>1$.
The proof can be achieved by the induction on $n_r$.
By our induction assumption, 
$Li^{(p),\star,\varpi}_{n_1,\dots,n_{r}-1}(\xi_1,\dots,\xi_{r}; z)\in A^\varpi_{\textrm{Col}}({\bf P}^1\setminus S_{r})$
and also
$Li^{(p),\star,\varpi}_{n_1,\dots,n_r-1}(\xi_1,\dots,\xi_r; 0)=0$. 
Hence $Li^{(p),\star,\varpi}_{n_1,\dots,n_r-1}(\xi_1,\dots,\xi_r; t)$
has a zero at $t=0$.
Therefore the integrand on the right-hand side of \eqref{intLi2} has no pole at $t=0$.
The integration \eqref{intLi2} starting from $0$ makes sense and thus
we have 
$Li^{(p),\star,\varpi}_{n_1,\dots,n_r}(\xi_1,\dots,\xi_r; z)\in A^\varpi_{\textrm{Col}}({\bf P}^1\setminus S_{r}).$
\end{proof}

It should be noted that the restriction of the $p$-adic TMSPL
$Li^{(p),\star,\varpi}_{n_1,\dots,n_r}(\xi_1,\dots,\xi_r ;z)$
to
${\bf P}^1({\mathbb C}_p) -\ ]{S_r}\setminus\{\overline{0}\} [$
does not depend on any choice of  the branch $\varpi\in\mathbb{C}_p$,
which can be proved in the same way as \cite[Proposition 3.34]{FKMT2}.

In particular, we remind that it is shown in \cite[Theorem-Definiton 3.38]{FKMT2} that,
for $\rho_1,\dots,\rho_r\in\mu_p$ and
$\xi_1,\dots,\xi_r\in\mu_c$ with $(c,p)=1$ and
$$\xi_1\cdots\xi_r\neq 1, \quad \xi_2\cdots\xi_r\neq 1,\quad
\dots, \quad  \xi_{r-1}\xi_r\neq 1,\quad \xi_r\neq 1,$$
the special value of
$Li^{(p),\star,\varpi}_{n_1,\dots,n_r}(\rho_1\xi_1,\dots,\rho_r\xi_r; z)$
at $z=1$
is independent of the choice of $\varpi$
and the value is denoted by
$Li^{(p),\star}_{n_1,\dots,n_r}(\rho_1\xi_1,\dots,\rho_r\xi_r)$
in short, is called the {\bf $p$-adic twisted multiple $L$-star value}.

The following result is an extension of \cite[Theorem 3.36]{FKMT2} to the case of indices 
with arbitrary integers, where we give a relationship between
our $p$-adic rigid TMSPL
$\ell^{(p),\star}_{n_1,\dots,n_r}(\xi_1,\dots,\xi_{r}; z)$
and our $p$-adic TMSPL
$Li^{(p),\star,\varpi}_{n_1,\dots,n_r}(\xi_1,\dots,\xi_r ;z)$. 
\begin{prp}\label{ell-Li reformulation prop}
Fix a branch  $\varpi\in\mathbb C_p$.
Let $n_1,\dots,n_r\in \mathbb{Z}$,
$\xi_1,\dots,\xi_{r}\in\mathbb{C}_p$
with $|\xi_j|_p= 1$ $(1\leqslant  j\leqslant  r)$.
The equality
\begin{align}
\ell^{(p),\star}_{n_1,\dots,n_r}&(\xi_1,\dots,\xi_r; z) 
=Li^{(p),\star,\varpi}_{n_1,\dots,n_r}(\xi_1,\dots, \xi_r; z) \label{ell-Li reformulation}  \\
&+\sum_{d=1}^r\left(-\frac{1}{p}\right)^d
\sum_{1\leqslant  i_1<\cdots<i_d\leqslant  r}\sum_{\rho_{i_1}^p=1}\cdots\sum_{\rho_{i_d}^p=1}
Li^{(p),\star,\varpi}_{n_1,\dots,n_r}\Bigl(\bigl((\prod_{l=1}^d\rho_{i_l}^{\delta_{i_lj}})\xi_j\bigr); z\Bigr)
\notag
\end{align}
holds for $z\in{\bf P}^1({\mathbb C}_p) - ]{S_r}\setminus\{\overline{0}\} [$, 
where $\delta_{ij}$ is the Kronecker delta. 
\end{prp}

\begin{proof}
By using the power series expansions 
\eqref{series expression for partial double}
and \eqref{series expression}, 
direct calculations show that the equality holds on $]\bar 0[$.
By Corollary \ref{rigidness 3},
the left-hand side belongs to 
$A^\dag({\bf P}^1
\setminus{S_r})$
($\subset A^\varpi_{\textrm{Col}}({\bf P}^1 \setminus{S_r})$), 
while by 
Proposition \ref{Coleman function theorem}, 
the right-hand side  belongs to $A^\varpi_{\textrm{Col}}({\bf P}^1  \setminus{S_r})$.
Therefore
by the coincidence principle  (consult \cite[Proposition 3.27]{FKMT2}),
the equality holds on the whole space of
${\bf P}^1({\mathbb C}_p) -\ ]{S_r}\setminus\{\overline{0}\} [$,
in fact, on an affinoid bigger than the space.
\end{proof}

Our  main theorem in this section is the following,
which could be regarded as  an extension of \cite[Theorem 3.41]{FKMT2}
to the case of indices 
with arbitrary integers 
and might be also regarded as  an extension of \cite[Theorem 2.1]{FKMT2}
to the case of  indices 
with arbitrary integers 
in the special case of  $\gamma_1=\cdots=\gamma_r=1$.

\begin{thm}\label{L-Li theorem-2}
For  $n_1,\dots,n_r\in \mathbb{Z}$
and $c\in \mathbb{N}_{>1}$ with $(c,p)=1$,
\begin{align}
& L_{p,r} (n_1,\dots,n_r;\omega^{-n_1},\dots,\omega^{-n_r};c) \label{val-pMLF} \\
&=
\underset{\xi_1\neq 1}{\sum_{\xi_1^c=1}} \cdots
\underset{\xi_r\neq 1}{\sum_{\xi_r^c=1}} 
Li^{(p),\star}_{n_1,\dots,n_r}\left(\frac{\xi_1}{\xi_2},\frac{\xi_2}{\xi_3},\dots,\frac{\xi_{r}}{\xi_{r+1}}\right)   \notag \\
&+\sum_{d=1}^r\left(-\frac{1}{p}\right)^d
\sum_{1\leqslant  i_1<\cdots<i_d\leqslant  r}\sum_{\rho_{i_1}^p=1}\cdots\sum_{\rho_{i_d}^p=1}
\underset{\xi_1\neq 1}{\sum_{\xi_1^c=1}} \cdots
\underset{\xi_r\neq 1}{\sum_{\xi_r^c=1}} 
Li^{(p),\star}_{n_1,\dots,n_r}\Bigl(\bigl(\frac{\prod_{l=1}^d\rho_{i_l}^{\delta_{i_lj}}\xi_j}{\xi_{j+1}}\bigr)\Bigr), \notag
\end{align}
where we put $\xi_{r+1}=1$.
\end{thm}

\begin{proof}
It follows from 
Proposition \ref{L-ell prop} and Proposition \ref{ell-Li reformulation prop}.
\end{proof}

\begin{rem}\label{Rem-fin}
From \eqref{MPL-03}, we have
\begin{equation}
\underset{\xi_1\neq 1}{\sum_{\xi_1^c=1}} \cdots
\underset{\xi_r\neq 1}{\sum_{\xi_r^c=1}} \zeta_r((n_j);(\xi_j);(1))
=\underset{\xi_1\neq 1}{\sum_{\xi_1^c=1}} \cdots
\underset{\xi_r\neq 1}{\sum_{\xi_r^c=1}} Li_{n_1,\ldots,n_r}\left(\frac{\xi_1}{\xi_2},\frac{\xi_{2}}{\xi_3},\ldots,\frac{\xi_{r}}{\xi_{r+1}}\right), \label{MPL-04}
\end{equation}
where $\xi_{r+1}=1$. Similarly, we obtain
\begin{equation}
\underset{\xi_1\neq 1}{\sum_{\xi_1^c=1}} \cdots
\underset{\xi_r\neq 1}{\sum_{\xi_r^c=1}} \zeta_r^\star((n_j);(\xi_j);(1))
=\underset{\xi_1\neq 1}{\sum_{\xi_1^c=1}} \cdots
\underset{\xi_r\neq 1}{\sum_{\xi_r^c=1}} Li_{n_1,\ldots,n_r}^\star\left(\frac{\xi_1}{\xi_2},\frac{\xi_{2}}{\xi_3},\ldots,\frac{\xi_{r}}{\xi_{r+1}}\right), \label{MPL-04-2}
\end{equation}
where
\begin{align}\label{Def-EZL-zeta-star}                                                       
\zeta_r^\star((n_j);(\xi_j);(1))&=\sum_{0<k_1\leqslant \cdots\leqslant k_r}\frac{(\xi_1/\xi_2)^{k_1}\cdots (\xi_r/\xi_{r+1})^{k_r}}{k_1^{n_1}\cdots k_r^{n_r}},
\end{align}
with $\xi_{r+1}=1$ and 
\begin{align}\label{MPL-01-star}
Li_{n_1,\ldots,n_r}^\star(z_1,\ldots,z_r)&=\sum_{0<k_1\leqslant \cdots\leqslant k_r}\frac{z_1^{k_1}\cdots z_r^{k_r}}{k_1^{n_1}\cdots k_r^{n_r}}
\end{align}
for $(n_j)\in \mathbb{N}^r$ and $(z_j)\in \mathbb{C}^r$ with $|z_j|= 1$, 
which are star-versions of \eqref{Def-EZL-zeta} and \eqref{MPL-01}, respectively. 
Also \eqref{MPL-01-star} should be compared with \eqref{series expression}.
Note that 
Theorem \ref{L-Li theorem-2} can be regarded as a $p$-adic analogue of \eqref{MPL-04-2}. Therefore $L_{p,r} ((s_j);(\omega^{k_j});c)$ might be called the {\it $p$-adic multiple $L$-star function}.
\end{rem}

\begin{crl}\label{Cor-final}
For $n_1,\ldots,n_r\in \mathbb{N}_0$ 
and $c\in \mathbb{N}_{>1}$ with $(c,p)=1$, 
\begin{align}
& L_{p,r}(-n_1,\ldots,-n_r;\omega^{n_1},\ldots,\omega^{n_r};c)  \label{Cor-main} \\
& =\sum_{\xi_1^c=1 \atop \xi_1\not=1}\cdots\sum_{\xi_r^c=1 \atop \xi_r\not=1}\aaa((n_j);(\xi_j))\notag \\
& +\sum_{d=1}^{r}\left(-\frac{1}{p}\right)^d \sum_{1\leqslant i_1<\cdots<i_d\leqslant r}\sum_{\rho_{i_1}^p=1}\cdots\sum_{\rho_{i_d}^p=1}\sum_{\xi_1^c=1 \atop \xi_1\not=1}\cdots\sum_{\xi_r^c=1 \atop \xi_r\not=1}\aaa((n_j);((\prod_{j\leqslant i_l}\rho_{i_l})\xi_j)),\notag
\end{align}
where $\{\aaa((n_j);( \xi_j))\}$
 are certain 
{ twisted multiple Bernoulli numbers}
defined by 
\begin{align}
    &\frac{\xi_1 \exp\left( \sum_{\nu=1}^r t_\nu\right)}{1-\xi_1 \exp\left( \sum_{\nu=1}^r t_\nu\right)} \prod_{j=2}^{r} \frac{1}{1-\xi_j \exp\left( \sum_{\nu=j}^r t_\nu\right)}\label{Fro-def-r} \\
&  \quad   =\sum_{n_1=0}^\infty
    \cdots
    \sum_{n_r=0}^\infty
    \aaa((n_j);( \xi_j))
    \frac{t_1^{n_1}}{n_1!}
    \cdots
    \frac{t_r^{n_r}}{n_r!}.\notag
\end{align}
\end{crl}

\begin{proof}
We first show the following result which can be proved by the same method as in the proof of \cite[Lemma 5.9]{KT}. 
For $z\in ]\bar{0}[$, we obtain from the definition \eqref{series expression}   that 
\begin{align}
& \sum_{n_1=0}^\infty \cdots \sum_{n_r=0}^\infty Li^{(p),\star}_{-n_1,\dots,-n_r}\left(\frac{\xi_1}{\xi_2},\frac{\xi_2}{\xi_3},\dots,\frac{\xi_{r}}{\xi_{r+1}};z\right) \frac{t_1^{n_1}\cdots t_r^{n_r}}{n_1!\cdots n_r!}\label{gene-ft-Li}\\
& \quad =\frac{\xi_1ze^{\sum_{\nu=1}^{r}t_\nu}}{1-\xi_1ze^{\sum_{\nu=1}^{r}t_\nu}}\prod_{j=2}^{r}\frac{1}{1-\xi_jze^{\sum_{\nu=j}^{r}t_\nu}}\notag
\end{align}
(cf. \cite[(5.16)]{KT}). 
Since 
$Li^{(p),\star}_{-n_1,\dots,-n_r}\left({\xi_1}/{\xi_2},{\xi_2}/{\xi_3},\dots,{\xi_{r}}/{\xi_{r+1}};z\right)$
is a rational function in $z$, we can let $z\to 1$ on the both sides of \eqref{gene-ft-Li}. Hence it follows from \eqref{Fro-def-r} that 
\begin{equation}\label{Li-val}
Li^{(p),\star}_{-n_1,\dots,-n_r}\left(\frac{\xi_1}{\xi_2},\frac{\xi_2}{\xi_3},\dots,\frac{\xi_{r}}{\xi_{r+1}}\right)=\aaa((n_j);( \xi_j))\quad ((n_j)\in \mathbb{N}_0^r).
\end{equation}
Therefore we can see that the right-hand side of \eqref{val-pMLF} coincides with the right-hand side of \eqref{Cor-main}. This completes the proof.
\end{proof}

\begin{rem}
It should be emphasized that \eqref{Cor-main} with replacing $\aaa((n_j);(\xi_j))$ by $\aa((n_j);(\xi_j))$ 
defined by 
\begin{align*}
    &\prod_{j=1}^{r} \frac{1}{1-\xi_j \exp\left( \sum_{\nu=j}^r t_\nu\right)}
=\sum_{n_1=0}^\infty
    \cdots
    \sum_{n_r=0}^\infty
    \aa((n_j);( \xi_j))
    \frac{t_1^{n_1}}{n_1!}
    \cdots
    \frac{t_r^{n_r}}{n_r!}\notag
\end{align*}
(see \cite[Definition 1.4]{FKMT2}) is also valid; in fact, it is \cite[Theorem 2.1]{FKMT2}. 

\end{rem}

Finally, we consider the case $r=1$. Since
\begin{align*}
&\sum_{\xi^c=1 \atop \xi\neq 1}\frac{\xi e^{t}}{1-\xi e^{t}}=\frac{e^t}{e^t-1}-\frac{ce^{ct}}{e^{ct}-1}=\sum_{n=0}^\infty (1-c^{n+1})B_{n+1}\frac{t^n}{n!}+(1-c),\\
&\sum_{\rho^p=1}\sum_{\xi^c=1 \atop \xi\neq 1}\frac{\rho\xi e^{t}}{1-\rho\xi e^{t}}=\sum_{\rho^p=1}\left\{\frac{\rho e^t}{\rho e^t-1}-\frac{c\rho^c e^{ct}}{\rho^c e^{ct}-1}\right\}=\frac{pe^{pt}}{e^{pt}-1}-\frac{cp e^{cpt}}{e^{cpt}-1}\\
& \qquad =\sum_{n=0}^\infty (1-c^{n+1})p^{n+1}B_{n+1}\frac{t^n}{(n+1)!}+(1-c)p,
\end{align*}
we have 
\begin{equation*}
\sum_{\xi^c=1 \atop \xi\neq 1}\aaa(n;\xi)=
\begin{cases}
(1-c^{n+1})\frac{B_{n+1}}{n+1}& (n>0),\\
\frac{1-c}{2} & (n=0),
\end{cases}
\end{equation*}
\begin{equation*}
\sum_{\xi^c=1 \atop \xi\neq 1}\sum_{\rho^p=1}\aaa(n;\rho\xi)=
\begin{cases}
(1-c^{n+1})p^{n+1}\frac{B_{n+1}}{n+1}& (n>0),\\
\frac{(1-c)p}{2} & (n=0).
\end{cases}
\end{equation*}
Hence \eqref{Cor-main} implies that
\begin{align*}
L_{p,1}(-n;\omega^{n};c)&=\sum_{\xi^c=1 \atop \xi\neq 1}\aaa(n;\xi)
-\frac{1}{p}\sum_{\rho^p=1}\sum_{\xi^c=1 \atop \xi\not=1}\aaa(n;\rho\xi)\notag\\
& =
\begin{cases}
(1-c^{n+1})(1-p^{n})\frac{B_{n+1}}{n+1}& (n>0),\\
0 & (n=0).
\end{cases}
\end{align*}
By \eqref{KL-pMLF}, this can be rewritten as the Kubota-Leopoldt formula (\cite[Theorem 5.11]{Wa}):
\begin{align}\label{KL-F}
& L_{p}(1-n;\omega^{n})=-(1-p^{n-1})\frac{B_{n}}{n}\quad (n\in \mathbb{N}).
\end{align}
On the other hand, combining \eqref{val-pMLF} in the case $r=1$ and \eqref{KL-pMLF}, we obtain the Coleman formula (\cite{C}):
\begin{equation}\label{Coleman-F}
L_p(n;\omega^{1-n})=\left(1-\frac{1}{p^n}\right)Li_n^{(p),\star}(1)\quad (n\in \mathbb{N})
\end{equation}
(see \cite[Example 3.42]{FKMT2}). Therefore Theorem \ref{L-Li theorem-2} can be regarded as a generalization of both \eqref{KL-F} and \eqref{Coleman-F}.

\ 

\end{document}